\documentclass [12pt]{amsart}
\usepackage[utf8]{inputenc}
\pdfoutput=1

\usepackage{amsmath,amssymb,amsthm,amsfonts}
\usepackage{mathtools}%

\usepackage[main=english,french]{babel}%
\usepackage{comment}%
\usepackage[bookmarksdepth=4]{hyperref}%
\usepackage{bbm}%
\usepackage{mathrsfs}%

\usepackage{tikz,graphicx,color}
\usepackage{tikz-cd}%
\usepackage{tikz-3dplot}%
\usetikzlibrary{calc}%
\usetikzlibrary{arrows}%
\usetikzlibrary{shapes}%
\usetikzlibrary{patterns}%
\usetikzlibrary{positioning}%
\usetikzlibrary{arrows.meta}
\usetikzlibrary{decorations.markings}
\usepackage{epstopdf}%
\pdfpageattr{/Group <</S /Transparency /I true /CS /DeviceRGB>>}

\usepackage[arrow]{xy}%
\usepackage{diagbox}%
\usepackage{subfig}%
\usepackage{arcs}%
\usepackage{xcolor}%
\usepackage{tabu}
\usepackage{booktabs}%

\usepackage[margin=1.01in]{geometry}

\usepackage{soul}
\usepackage{accents}
\usepackage{enumitem}
\usepackage{letltxmacro}
\usepackage{thmtools,etoolbox}

\def\myarabic#1{\normalfont(\roman{#1})}
\newlist{theoremlist}{enumerate}{1}
\setlist[theoremlist]{label=\myarabic{theoremlisti},ref={\myarabic{theoremlisti}},itemindent=0pt,labelindent=0pt,
leftmargin=*,noitemsep}

\makeatletter
\renewcommand{\p@theoremlisti}{\perh@ps{\thetheorem}}
\protected\def\perh@ps#1#2{\textup{#1#2}}
\newcommand{\itemrefperh@ps}[2]{\textup{#2}}
\newcommand{\itemref}[1]{\begingroup\let\perh@ps\itemrefperh@ps\ref{#1}\endgroup}
\makeatother

\usepackage{nameref,hyperref}
\usepackage[capitalize]{cleveref}

\newtheorem{theorem}{Theorem}[section]
\newtheorem*{statement}{Theorem}
\newtheorem{lemma}[theorem]{Lemma}
\newtheorem{proposition}[theorem]{Proposition}
\newtheorem{corollary}[theorem]{Corollary}

\theoremstyle{definition}
\newtheorem{remark}[theorem]{Remark}
\theoremstyle{definition}

\theoremstyle{definition}

\theoremstyle{definition}
\newtheorem{example}[theorem]{Example}

\crefname{figure}{Figure}{Figures}

\addtotheorempostheadhook[theorem]{\crefalias{theoremlisti}{theorem}}
\addtotheorempostheadhook[lemma]{\crefalias{theoremlisti}{lemma}}
\addtotheorempostheadhook[proposition]{\crefalias{theoremlisti}{proposition}}
\addtotheorempostheadhook[corollary]{\crefalias{theoremlisti}{corollary}}

\def\Acal{\mathcal{A}}

\def\gbf{\mathbf{g}}\def\xbf{\mathbf{x}}

\def\one{{\mathbbm{1}}}
\def\C{\mathbb{C}}
\def\R{\mathbb{R}}

\def\Z{\mathbb{Z}}

\newcommand\parr[1]{{({#1})}}
\def\<{{\langle}}
\def\>{{\rangle}}

\def\la{{\lambda}}

\def\diag{{ \operatorname{diag}}}

\def\wt{\operatorname{wt}}
\def\inv{\operatorname{inv}}

\def\SL{\operatorname{SL}}

\def\Gr{\operatorname{Gr}}

\def\Z{{\mathbb Z}}
\def\R{{\mathbb R}}
\def\Gr{{\rm Gr}}

\def\diag{{\rm diag}}

\newcommand{\smat}[1]{\left[\begin{smallmatrix}
      #1
    \end{smallmatrix}\right]}

\def\X{\Xcal}

\def\u{h}

\def\bs{\backslash}

\def\t{{\mathbf{t}}}

\def\d#1{\dot{#1}}
\def\ds{\d{s}}
\def\dw{\d{w}}

\def\dv{\d{v}}

\def\X{\accentset{\circ}{X}}

\def\Rich_#1^#2{\accentset{\circ}{R}_{#1,#2}}
\def\Richcl_#1^#2{R_{#1,#2}}
\def\RRich_#1^#2{\accentset{\circ}{R}_{#1,#2}^\R}%
\def\RRichcl_#1^#2{R_{#1,#2}^\R}%
\def\Rtp_#1^#2{R_{#1,#2}^{>0}}
\def\Rtnn_#1^#2{R_{#1,#2}^{\geq0}}

\def\PR_#1^#2{\accentset{\circ}{\Pi}_{#1,#2}}%
\def\PRtp_#1^#2{\Pi_{#1,#2}^{>0}}%
\def\PRtnn_#1^#2{\Pi_{#1,#2}^{\geq0}}%
\def\PRcl_#1^#2{\Pi_{#1,#2}}%
\def\PRR_#1^#2{\accentset{\circ}{\Pi}_{#1,#2}^\R}%
\def\PRRcl_#1^#2{\Pi_{#1,#2}^\R}%

\def\sf{\operatorname{sf}}

\def\bt{{\mathbf{t}}}
\def\bw{{\mathbf{w}}}
\def\bv{{\mathbf{v}}}

\def\maxx{{\operatorname{max}}}

\def\hjmap{\kappa}
\def\hjmp_#1{\hjmap_{#1}}

\def\Uom_#1{U^{\diamond,-}_{#1}}

\def\Rsf_#1^#2{R_{#1,#2}^{\sf}}

\def\xrasim{\xrightarrow{\sim}}

\def\Richaff_#1^#2{\accentset{\circ}{\mathcal{R}}_{#1}^{#2}}

\def\bs{\backslash}

\def\Cast{\C^\ast}

\def\Povar_#1{\accentset{\circ}{\Pi}_{#1}}
\def\Povarcl_#1{\Pi_{#1}}
\def\RPovar_#1{\accentset{\circ}{\Pi}^\R_{#1}}
\def\RPovarcl_#1{\Pi^\R_{#1}}
\def\Povtp_#1{\Pi_{#1}^{>0}}
\def\Povtnn_#1{\Pi_{#1}^{\geq0}}

\def\Star_#1{\operatorname{Star}_{#1}}
\def\Startnn_#1{\operatorname{Star}^{\geq0}_{#1}}

\def\Link{\operatorname{Lk}}
\def\Lkx_#1{\Link_{#1}}
\def\Lkxx_#1^#2{\accentset{\circ}{\Link}_{#1}^{#2}}

\def\Lktxx_#1^#2{\Link^{>0}_{#1,#2}}
\def\Starxx_#1^#2{\operatorname{Star}_{#1,#2}}
\def\Startxx_#1^#2{\operatorname{Star}^{\geq0}_{#1,#2}}

\def\sctnn_#1{\sc^{\geq0}_{#1}}

\def\sctp_#1^#2{\sc^{>0}_{#1,#2}}

\def\eps{\varepsilon}
\def\Seps_#1{S_{#1}}

\def\Lktpe_#1^#2{\Link^{>0}_{#1,#2}}
\def\Lktnne_#1{\Link^{\geq0}_{#1}}

\def\Lktp_#1^#2{\Link^{>0}_{#1,#2}}
\def\Lktnn_#1{\Link^{\geq0}_{#1}}

\def\sc{Z}

\def\sco_#1^#2{\accentset{\circ}{\sc}_{#1,#2}}

\def\sccl_#1^#2{\sc_{#1}^{#2}}

\def\Y{\mathcal{Y}}
\def\Yo_#1{\accentset{\circ}{\Y}_{#1}}
\def\Ycl_#1{\Y_{#1}}

\def\Ytp_#1{\Y_{#1}^{>0}}

\def\strg(#1){\normg{#1}}

\def\normg#1{\|#1\|}

\def\Jo{J_\bv^\circ}

\def\Spec{\operatorname{Spec}}
\def\int{{\operatorname{init}}}

\def\NW{\operatorname{NW}}
\def\WNW{\overline{\NW}}
\def\shift{\operatorname{shift}}
\let\xto\xrightarrow
\def\xxto#1{\xhookrightarrow{#1}}
\def\tQ{{\tilde Q}}

\title[Positroid varieties and cluster algebras]{Positroid varieties and cluster algebras}

\author{Pavel Galashin}
\address{Department of Mathematics, University of California, Los Angeles, 520 Portola Plaza,
Los Angeles, CA 90025, USA}
\email{\href{mailto:galashin@math.ucla.edu}{galashin@math.ucla.edu}}

\author{Thomas Lam}
\address{Department of Mathematics, University of Michigan, 2074 East Hall, 530 Church Street, Ann Arbor, MI 48109-1043, USA}
\email{\href{mailto:tfylam@umich.edu}{tfylam@umich.edu}}
\thanks{P.G.\ was supported by an Alfred P. Sloan Research Fellowship and by the National Science Foundation under Grants No.~DMS-1954121 and No.~DMS-2046915. T.L.\ was supported by a von Neumann Fellowship from the Institute for Advanced Study and by Grants No.~DMS-1464693 and No.~DMS-1953852 from the National Science Foundation.}

\subjclass[2010]{
  Primary:
  13F60. %
  Secondary:
  14M15. %
}

\keywords{Cluster algebra, Grassmannian, positroid variety, total positivity, preprojective algebra, twist map.}

\date{\today}

\def\QD{Q_D}
\def\QDp{Q_{D'}}
\def\AQD{\Acal(\QD)}
\def\AQT{\Acal(\tQ)}
\def\PJ{P^J_-}



\def\rW{r_{\leftarrow}}
\def\rN{r_{\uparrow}}
\def\rNW{r'}

\begin{document}

\begin{abstract}
We show that the coordinate ring of an open positroid variety coincides with the cluster algebra associated to a Postnikov diagram. This confirms conjectures of Postnikov, Muller--Speyer, and Leclerc, and generalizes results of Scott and Serhiyenko--Sherman-Bennett--Williams.
\end{abstract}

\selectlanguage{english}

\numberwithin{equation}{section}

\def\Vi#1{v^{\parr{#1}}}
\def\Vii#1{\bar v^{\parr{#1}}}
\def\Wi#1{w^{\parr{#1}}}
\def\vi#1{v_{\parr{#1}}}
\def\wi#1{w_{\parr{#1}}}

\maketitle

Positroid varieties are subvarieties of the Grassmannian that first appeared in the study of total positivity and Poisson geometry \cite{Lus, Pos, BGY, KLS}.  In this paper we establish the following result; see \cref{thm:main}.
\begin{statement}
The coordinate ring $\C[\PR_v^w]$ of an open positroid variety $\PR_v^w$ is a cluster algebra.  
\end{statement}

For the top-dimensional open positroid variety, this is due to Scott \cite{Sco}, a result that motivated much of the subsequent work. Combinatorially, positroid varieties are parametrized by Postnikov diagrams, and each such diagram gives rise to a quiver whose vertices are labeled by Pl\"ucker coordinates on the Grassmannian; see~\cite{Pos,Sco}. This data gives rise to a cluster algebra of~\cite{FZ} whose cluster variables are rational functions on the Grassmannian, and since the work of Scott, it has been expected that this cluster algebra coincides with the coordinate ring of $\PR_v^w$. This conjecture was made explicit by Muller and Speyer \cite[Remark 4.6]{MStwist}, and was established recently in the special case of Schubert varieties by Serhiyenko--Sherman-Bennett--Williams~\cite{SSBW}.  Another closely related conjecture was given by Leclerc~\cite{Lec}, who constructed a cluster subalgebra of $\C[\PR_v^w]$ using representations of preprojective algebras.  We show (\cref{cor:same}) that these two cluster structures coincide. These cluster structures have also been compared in~\cite{SSBW}; our work differs from theirs by switching from a left-sided to a right-sided quotient for the flag variety, i.e., from $B_-\bs G$ to $G/B_-$; see \cref{rmk:easy_hard}. 

Leclerc's conjectures and results apply in the more general setting of open Richardson varieties.  We hope to return to cluster structures of open Richardson varieties in the future. Some other closely related cluster structures include double Bruhat cells \cite{BFZ,GY}, partial flag varieties \cite{GLS1}, and unipotent groups \cite{GLS2}.

Combining our main result with the well-developed machinery of cluster algebras has many consequences for the structure of open positroid varieties; see e.g. the introduction of~\cite{SSBW}. For instance,  the existence of a green-to-red sequence \cite{FS}, together with the constructions of \cite{GHKK} endow $\C[\PR_v^w]$ with a basis of \emph{theta functions} with positive structure constants.  Additionally, the results of \cite{LS} imply that $H^*(\PR_v^w,\C)$ satisfies the curious Lefschetz property, which has implications for  extension groups of certain Verma modules that we aim to explore in future work. 

Finally, we show that the totally nonnegative part $\PRtp_v^w$ of $\PR_v^w$ (as defined by~\cite{Lus,Pos}) is precisely the subset of $\PR_v^w$ where all cluster variables take positive real values; see \cref{cor:tnn}.

\subsection*{Acknowledgements.} We are grateful to Khrystyna Serhiyenko, Melissa Sherman-Bennett, and Lauren Williams for their comments on an earlier version of this manuscript and for conversations regarding the results of~\cite{SSBW}. These conversations motivated the start of this project and inspired the results in \cref{sec:Leclerc_cluster_algebra}. In addition, we thank Melissa Sherman-Bennett for pointing out a sign issue in \cref{lem:rotation}. 
The second author thanks David Speyer for discussions that led to Proposition \ref{prop:fullrank}. Finally, we are grateful to the anonymous referee for their careful reading of the text.

\subsection*{Outline}
We discuss the combinatorics of Le-diagrams in \cref{sec:Le_cluster}. The cluster algebra $\AQD$ coming from a Le-diagram $D$ consists of some rational functions on the Grassmannian. As we discuss in \cref{sec:main}, in order to prove our main result, one needs to show two inclusions: $\AQD\subseteq \C[\PR_v^w]$ and $\AQD\supseteq \C[\PR_v^w]$. For the first inclusion, we rely on the results of Leclerc~\cite{Lec}. In particular, following ideas of~\cite{SSBW}, we show in \cref{sec:Leclerc_cluster_algebra} that the cluster algebra of~\cite{Lec} is isomorphic to $\AQD$ (i.e., they have isomorphic quivers). We then prove the first inclusion $\AQD\subseteq \C[\PR_v^w]$ in \cref{sec:positroid-varieties}; see \cref{cor:Lec}. We show the second inclusion $\AQD\supseteq \C[\PR_v^w]$ in \cref{sec:surj}, using the results of Muller--Speyer~\cite{MStwist,MSLA}, of Muller \cite{Mul}, and of Berenstein--Fomin--Zelevinsky~\cite{BFZ}.

\medskip
Throughout the paper, we fix a positive integer $n$, and an integer $k \in [n]:=\{1,2,\ldots,n\}$. For $a,b\in\Z$, we let $[a,b]:=\{a,a+1,\dots,b\}$ if $a\leq b$, and $[a,b]:=\emptyset$ otherwise.

\section{Le-diagram cluster algebra}\label{sec:Le_cluster}

Let $W = S_n$ be the symmetric group on $n$ letters. For $i\in [n-1]$, let $s_i\in W$ denote the simple transposition of $i$ and $i+1$. Every permutation $w \in W$ can be written as a reduced word $w=s_{i_1}\cdots s_{i_m}$ (where $m=\ell(w)$ is the length of $w$).  In this case, $\bw:=(i_1,\dots,i_m)$ is called a \emph{reduced expression} for $w$. We multiply permutations right-to-left; in particular, for $j\in [n]$ and $w=s_{i_1}\cdots s_{i_m}$, we let $w(j):=s_{i_1}(\dots(s_{i_m}(j))\dots )$. For $A\subset[n]$, we denote $wA:=\{w(a)\mid a\in A\}$. 

Let $J = [n] \setminus \{k\}$, and denote by $W^J \subset W$ the set of \emph{$k$-Grassmannian permutations}, i.e., permutations $w\in W$ satisfying $w(1)<\dots<w(k)$ and $w(k+1)<\dots<w(n)$.  In other words, we have $w\in W^J$ if and only if $w=1$ or each reduced word for $w$ ends with $s_k$. 

Let $Q^J$ denote the set of pairs $(v,w)$ where  $w \in W^J$ and $v \leq w$ in the Bruhat order on $W$. The elements of $Q^J$ label positroid varieties; see \cref{sec:pos}. By~\cite[Lemma~3.5]{MR}, every reduced expression $\bw=(i_1,\dots,i_m)$ for $w$ contains a unique ``rightmost'' reduced subexpression $\bv$ for $v$, called the \emph{positive distinguished subexpression}. We let $\Jo\subset [m]$ denote the set of indices \emph{not} used in $\bv$.

\def\fdot{\bullet}

\subsection{Le-diagrams and subexpressions}\label{sec:QJtoLe}
We use English notation for Young diagrams and label their boxes in matrix notation.  A \emph{Le-diagram} $D$ is a Young diagram $\lambda$, contained in a $k \times (n-k)$ rectangle, together with a filling of some of its boxes with dots, satisfying the following condition: if a box $b$ is both below a dot and to the right of a dot, then $b$ must contain a dot.

\def\borderInnerSep{2pt}
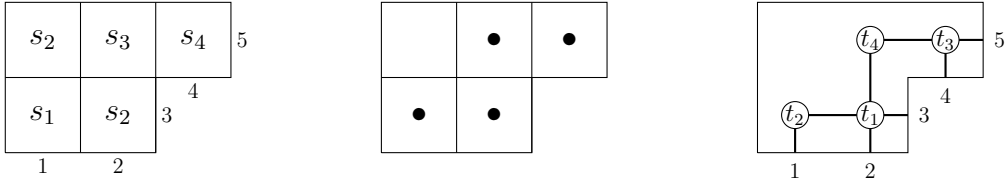
\begin{figure}
\begin{tikzpicture}
\draw (0,0)-- (0,2)--(3,2)--(3,1)--(2,1)--(2,0)--(0,0);
\draw (1,0)--(1,2);
\draw (2,0)--(2,2);
\draw (0,1)--(2,1);

\def\nodescl{0.7}
\node[anchor=north, inner sep=\borderInnerSep] (A) at (0.5,0) {\scalebox{\nodescl}{$1$}};
\node[anchor=north, inner sep=\borderInnerSep] (A) at (1.5,0) {\scalebox{\nodescl}{$2$}};
\node[anchor=north, inner sep=\borderInnerSep] (A) at (2.5,1) {\scalebox{\nodescl}{$4$}};
\node[anchor=west, inner sep=\borderInnerSep] (A) at (2,0.5) {\scalebox{\nodescl}{$3$}};
\node[anchor=west, inner sep=\borderInnerSep] (A) at (3,1.5) {\scalebox{\nodescl}{$5$}};
\node (A) at (0.5,0.5) {$s_1$};
\node (A) at (1.5,0.5) {$s_2$};
\node (A) at (0.5,1.5) {$s_2$};
\node (A) at (1.5,1.5) {$s_3$};
\node (A) at (2.5,1.5) {$s_4$};
\begin{scope}[shift={(5,0)}]
\draw (0,0)-- (0,2)--(3,2)--(3,1)--(2,1)--(2,0)--(0,0);
\draw (1,0)--(1,2);
\draw (2,0)--(2,2);
\draw (0,1)--(2,1);
\def\nodescl{0.7}
\node (A) at (0.5,0.5) {$\fdot$};
\node (A) at (1.5,0.5) {$\fdot$};
\node (A) at (1.5,1.5) {$\fdot$};
\node (A) at (2.5,1.5) {$\fdot$};
\end{scope}

\begin{scope}[shift={(10,0)}]

\def\nodescl{0.7}
\draw (0,0)-- (0,2)--(3,2)--(3,1)--(2,1)--(2,0)--(0,0);
\node[anchor=north] (A) at (0.5,0) {\scalebox{\nodescl}{$1$}};
\node[anchor=north] (A) at (1.5,0) {\scalebox{\nodescl}{$2$}};
\node[anchor=north] (A) at (2.5,1) {\scalebox{\nodescl}{$4$}};
\node[anchor=west] (A) at (2,0.5) {\scalebox{\nodescl}{$3$}};
\node[anchor=west] (A) at (3,1.5) {\scalebox{\nodescl}{$5$}};
\draw[line width=0.8pt] (0.5,0.5)--(2,0.5);
\draw[line width=0.8pt] (1.5,1.5)--(3,1.5);
\draw[line width=0.8pt] (1.5,1.5)--(1.5,0);
\draw[line width=0.8pt] (0.5,0.5)--(0.5,0);
\draw[line width=0.8pt] (2.5,1.5)--(2.5,1);
\node[draw,circle,scale=0.8,inner sep=0pt,fill=white] (A) at (0.5,0.5) {$t_2$};
\node[draw,circle,scale=0.8,inner sep=0pt,fill=white] (A) at (1.5,0.5) {$t_1$};
\node[draw,circle,scale=0.8,inner sep=0pt,fill=white] (A) at (2.5,1.5) {$t_3$};
\node[draw,circle,scale=0.8,inner sep=0pt,fill=white] (A) at (1.5,1.5) {$t_4$};
\end{scope}

\end{tikzpicture}
  \caption{\label{fig:Le} The Young diagram $\la$, Le-diagram $D$, and graph $G(D)$ corresponding to $(v,w) = (s_2,s_2s_1s_4s_3s_2)$.}
\end{figure}

We describe a well-known bijection~\cite[Section~20]{Pos} between elements of $Q^J$ and Le-diagrams. First, Grassmannian permutations $w \in W^J$ are in bijection with Young diagrams $\lambda \subseteq k \times (n-k)$:  placing $s_{k+j-i}$ into the box $(i,j)$ of $\lambda$, a reduced word for $w$ is obtained by reading the boxes from right to left along each row, starting from the lowest row.  The southeastern boundary edges of $\la$ are labeled $1,2,\ldots,n$ from bottom-left to top-right. Thus the southern boundary edges are labeled by the elements of $w[k+1,n]$.

Given $v \leq w$, we mark the letters \emph{not} used by the positive distinguished subexpression for $v$ with a dot, and this gives a Le-diagram denoted $D(v,w)$. For example, if $(v,w) = (s_2,s_2s_1s_4s_3s_2)$, we have the Young diagram $\la=(3,2)$ and the Le-diagram $D(v,w)$ in \cref{fig:Le}(left and middle). Note that $w[k+1,n]=w\{3,4,5\} = \{1,2,4\}$ are the labels of the southern boundary edges. 

Throughout the paper, we assume $(v,w)\in Q^J$ and denote $D:=D(v,w)$. We also fix a choice of $\bw=(i_1,\dots,i_m)$, $\bv$, and $\Jo$ as above.

\subsection{The graph $G(D)$}
 To a Le-diagram $D$ we associate a planar graph $G(D)$ embedded into the disk.  The boundary of $\lambda$ is taken to be the boundary of the disk, and boundary vertices are placed at the east and south boundary steps of $\lambda$, labeled $1,2,\ldots$ in counterclockwise order, starting from the southwest corner of $\la$.  From each dot in $D$, we draw a hook: one line going eastward, and one line going southward, until they hit the boundary.  The interior vertices of $G(D)$ correspond to the dots of $D$. Each dot of $D$ corresponds to an element $r\in \Jo$, and we label the associated vertex of $G(D)$ by $t_r$. The edges of $G(D)$ are horizontal or vertical line segments connecting two dots. See \cref{fig:Le}(right).

\subsection{Quiver}\label{Le_quiver}
A \emph{quiver} $Q$ is a directed graph without directed cycles of length $1$ or $2$. An \emph{ice quiver} is a quiver $Q$ such that each vertex of $Q$ is declared to be either \emph{frozen} or \emph{mutable}. We always assume that an ice quiver contains no arrows between frozen vertices. In this section, we explain how to associate an ice quiver $\QD$ to a Le-diagram $D=D(v,w)$. 

For each $r\in \Jo$, the vertex of $G(D)$ labeled by $t_r$ is the northwestern corner of some face of $G(D)$, and we label this face by $F_r$. Label the remaining face of $G(D)$ (the one adjacent to the northwestern boundary of $\la$) by $F_0$. Thus the neighborhood of any vertex of $G(D)$ looks as in \cref{fig:neigh}.
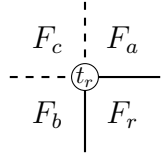
\begin{figure}
\begin{tikzpicture}
\draw[dashed, line width=0.8pt] (0,1)--(0,0);
\draw[dashed, line width=0.8pt] (-1,0)--(0,0);
\draw[line width=0.8pt] (0,0)--(1,0);
\draw[line width=0.8pt] (0,0)--(0,-1);
\node[draw,circle,scale=0.8,inner sep=0pt,fill=white] (A) at (0,0) {$t_r$};
\node (A) at (0.5,-0.5) {$F_r$};
\node (A) at (0.5,0.5) {$F_a$};
\node (A) at (-0.5,0.5) {$F_c$};
\node (A) at (-0.5,-0.5) {$F_b$};
\end{tikzpicture}
\caption{\label{fig:neigh} The neighborhood of a vertex in $G(D)$. The dashed lines may or may not be present, and $F_a,F_b,F_c$ are the labels of the corresponding faces. Thus some of them may coincide: we may have either $c=a$, or $c=b$, or both.}
\end{figure}

Construct a quiver $\QD$ whose vertices are $\{F_r\}_{r\in \Jo}$, i.e., the faces of $G(D)$ excluding $F_0$.  For each $r\in \Jo$,  depending on the local structure of $G(D)$ near the vertex labeled $t_r$ (cf. \cref{fig:neigh}), $\QD$ contains the arrows shown in \cref{fig:quiver_rules}.
\def\leop{0.4}
\begin{figure}
\begin{tikzpicture}
\def\qlw{1pt}
\def\qop{1}

\def\Finnersep{2pt}
\def\stp{3}

\node(A) at (0,0){
\begin{tikzpicture}
\draw[line width=0.8pt,opacity=\leop] (0,0)--(1,0);
\draw[line width=0.8pt,opacity=\leop] (0,0)--(0,-1);
\node[draw=white,circle,scale=0.8,inner sep=0pt,fill=white] (A) at (0,0) {\textcolor{white}{$t_r$}};
\node[draw,circle,scale=0.8,inner sep=0pt,fill=white,opacity=\leop] (A) at (0,0) {$t_r$};
\node[blue,inner sep=\Finnersep] (Fr) at (0.5,-0.5) {$F_r$};
\node[blue,inner sep=\Finnersep] (Fc) at (-0.5,0.5) {$F_c$};
\draw[->,blue,opacity=\qop,line width=\qlw] (Fr)--(Fc);
\end{tikzpicture}
};

\node(A) at (\stp,0){
\begin{tikzpicture}
\draw[line width=0.8pt,opacity=\leop] (-1,0)--(1,0);
\draw[line width=0.8pt,opacity=\leop] (0,0)--(0,-1);
\node[draw=white,circle,scale=0.8,inner sep=0pt,fill=white] (A) at (0,0) {\textcolor{white}{$t_r$}};
\node[draw,circle,scale=0.8,inner sep=0pt,fill=white,opacity=\leop] (A) at (0,0) {$t_r$};
\node[blue,inner sep=\Finnersep] (Fr) at (0.5,-0.5) {$F_r$};
\node[blue,inner sep=\Finnersep] (Fc) at (0,0.5) {$F_c$};
\node[blue,inner sep=\Finnersep] (Fb) at (-0.5,-0.5) {$F_b$};
\draw[->,blue,opacity=\qop,line width=\qlw] (Fr)--(Fb);
\end{tikzpicture}
};

\node(A) at (2*\stp,0){
\begin{tikzpicture}
\draw[line width=0.8pt,opacity=\leop] (0,0)--(1,0);
\draw[line width=0.8pt,opacity=\leop] (0,1)--(0,-1);
\node[draw=white,circle,scale=0.8,inner sep=0pt,fill=white] (A) at (0,0) {\textcolor{white}{$t_r$}};
\node[draw,circle,scale=0.8,inner sep=0pt,fill=white,opacity=\leop] (A) at (0,0) {$t_r$};
\node[blue,inner sep=\Finnersep] (Fr) at (0.5,-0.5) {$F_r$};
\node[blue,inner sep=\Finnersep] (Fa) at (0.5,0.5) {$F_a$};
\node[blue,inner sep=\Finnersep] (Fc) at (-0.5,0) {$F_c$};
\draw[->,blue,opacity=\qop,line width=\qlw] (Fr)--(Fa);
\end{tikzpicture}
};

\node(A) at (3*\stp,0){
\begin{tikzpicture}
\draw[line width=0.8pt,opacity=\leop] (-1,0)--(1,0);
\draw[line width=0.8pt,opacity=\leop] (0,1)--(0,-1);
\node[draw=white,circle,scale=0.8,inner sep=0pt,fill=white] (A) at (0,0) {\textcolor{white}{$t_r$}};
\node[draw,circle,scale=0.8,inner sep=0pt,fill=white,opacity=\leop] (A) at (0,0) {$t_r$};
\node[blue,inner sep=\Finnersep] (Fr) at (0.5,-0.5) {$F_r$};
\node[blue,inner sep=\Finnersep] (Fa) at (0.5,0.5) {$F_a$};
\node[blue,inner sep=\Finnersep] (Fc) at (-0.5,0.5) {$F_c$};
\node[blue,inner sep=\Finnersep] (Fb) at (-0.5,-0.5) {$F_b$};
\draw[->,blue,opacity=\qop,line width=\qlw] (Fr)--(Fa);
\draw[->,blue,opacity=\qop,line width=\qlw] (Fr)--(Fb);
\draw[->,blue,opacity=\qop,line width=\qlw] (Fc)--(Fr);
\end{tikzpicture}
};

\end{tikzpicture}
\caption{\label{fig:quiver_rules} Local rules for constructing a quiver from a Le-diagram.}
\end{figure}
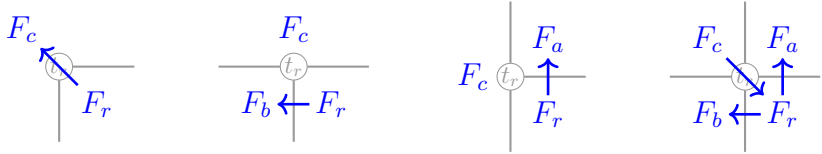

\def\bound{\partial \Jo}

\noindent The boundary (resp., interior) faces of $G(D)$ are designated as frozen (resp., mutable) vertices of $\QD$. We let $\bound\subset \Jo$  be the set of $r\in \Jo$ such that $F_r$ labels a boundary face of $G(D)$, i.e., is a frozen vertex of $\QD$.
 Some of the arrows in \cref{fig:quiver_rules} could connect frozen vertices, in which case we omit those arrows from $\QD$.

\begin{example}\label{ex:Le_big}
Let $k=3$, $n=6$, and $(v,w)=(s_2s_4,s_2s_1s_4s_3s_2s_5s_4s_3)$. The graph $G(D)$ is shown in \cref{fig:Le_quiver}(left), and the quiver $\QD$ is shown in \cref{fig:Le_quiver}(right). The only mutable vertex of $\QD$ is $F_8$, the vertices $F_1,F_2,F_3,F_4,F_6$ are frozen. Dashed arrows connect frozen vertices and therefore are not present in $\QD$.

\begin{figure}
\scalebox{0.9}{
\begin{tikzpicture}

\def\scl{1.3}
\def\Fscl{0.9}
\def\Fsep{4pt}
\def\arrlen{0.5}

\node(X) at (0,0){
\begin{tikzpicture}[scale=\scl]
\def\nodescl{0.7}
\draw (0,0)-- (2,0)--(2,1)--(3,1)--(3,3)--(0,3)--(0,0);
\node[anchor=north] (A) at (0.5,0) {\scalebox{\nodescl}{$1$}};
\node[anchor=north] (A) at (1.5,0) {\scalebox{\nodescl}{$2$}};
\node[anchor=north] (A) at (2.5,1) {\scalebox{\nodescl}{$4$}};
\node[anchor=west] (A) at (2,0.5) {\scalebox{\nodescl}{$3$}};
\node[anchor=west] (A) at (3,1.5) {\scalebox{\nodescl}{$5$}};
\node[anchor=west] (A) at (3,2.5) {\scalebox{\nodescl}{$6$}};
\draw[line width=0.8pt] (0.5,0.5)--(2,0.5);
\draw[line width=0.8pt] (1.5,1.5)--(3,1.5);
\draw[line width=0.8pt] (1.5,1.5)--(1.5,0);
\draw[line width=0.8pt] (0.5,0.5)--(0.5,0);
\draw[line width=0.8pt] (2.5,2.5)--(2.5,1);
\draw[line width=0.8pt] (0.5,2.5)--(3,2.5);
\draw[line width=0.8pt] (0.5,2.5)--(0.5,0);

\node[draw,circle,scale=0.8,inner sep=0pt,fill=white] (A2) at (0.5,0.5) {$t_2$};
\node[draw,circle,scale=0.8,inner sep=0pt,fill=white] (A1) at (1.5,0.5) {$t_1$};
\node[draw,circle,scale=0.8,inner sep=0pt,fill=white] (A3) at (2.5,1.5) {$t_3$};
\node[draw,circle,scale=0.8,inner sep=0pt,fill=white] (A4) at (1.5,1.5) {$t_4$};
\node[draw,circle,scale=0.8,inner sep=0pt,fill=white] (A6) at (2.5,2.5) {$t_6$};
\node[draw,circle,scale=0.8,inner sep=0pt,fill=white] (A8) at (0.5,2.5) {$t_8$};

\node[scale=\Fscl,blue,anchor=north west,inner sep=\Fsep] (F1) at (A1.south east) {$F_1$};
\node[scale=\Fscl,blue,anchor=north west,inner sep=\Fsep] (F2) at (A2.south east) {$F_2$};
\node[scale=\Fscl,blue,anchor=north west,inner sep=\Fsep] (F3) at (A3.south east) {$F_3$};
\node[scale=\Fscl,blue,anchor=north west,inner sep=\Fsep] (F4) at (A4.south east) {$F_4$};
\node[scale=\Fscl,blue,anchor=north west,inner sep=\Fsep] (F6) at (A6.south east) {$F_6$};
\node[scale=\Fscl,blue,anchor=north west,inner sep=\Fsep] (F8) at (A8.south east) {$F_8$};
\node[scale=\Fscl,blue,anchor=south east,inner sep=\Fsep] (outer) at (A8.north west) {$F_0$};
\end{tikzpicture}
};

\node(X) at (6,0){
\begin{tikzpicture}[scale=\scl]
\def\nodescl{0.7}
\begin{scope}[opacity=\leop]
\draw (0,0)-- (2,0)--(2,1)--(3,1)--(3,3)--(0,3)--(0,0);
\node[anchor=north] (A) at (0.5,0) {\scalebox{\nodescl}{$1$}};
\node[anchor=north] (A) at (1.5,0) {\scalebox{\nodescl}{$2$}};
\node[anchor=north] (A) at (2.5,1) {\scalebox{\nodescl}{$4$}};
\node[anchor=west] (A) at (2,0.5) {\scalebox{\nodescl}{$3$}};
\node[anchor=west] (A) at (3,1.5) {\scalebox{\nodescl}{$5$}};
\node[anchor=west] (A) at (3,2.5) {\scalebox{\nodescl}{$6$}};
\draw[line width=0.8pt] (0.5,0.5)--(2,0.5);
\draw[line width=0.8pt] (1.5,1.5)--(3,1.5);
\draw[line width=0.8pt] (1.5,1.5)--(1.5,0);
\draw[line width=0.8pt] (2.5,2.5)--(2.5,1);
\draw[line width=0.8pt] (0.5,2.5)--(3,2.5);
\draw[line width=0.8pt] (0.5,2.5)--(0.5,0);
\end{scope}
\node[circle,scale=0.8,inner sep=0pt,fill=white] (A2) at (0.5,0.5) {$\strut$};
\node[circle,scale=0.8,inner sep=0pt,fill=white] (A2) at (0.5,0.5) {$\strut$};
\node[circle,scale=0.8,inner sep=0pt,fill=white] (A1) at (1.5,0.5) {$\strut$};
\node[circle,scale=0.8,inner sep=0pt,fill=white] (A3) at (2.5,1.5) {$\strut$};
\node[circle,scale=0.8,inner sep=0pt,fill=white] (A4) at (1.5,1.5) {$\strut$};
\node[circle,scale=0.8,inner sep=0pt,fill=white] (A6) at (2.5,2.5) {$\strut$};
\node[circle,scale=0.8,inner sep=0pt,fill=white] (A8) at (0.5,2.5) {$\strut$};

\begin{scope}[opacity=\leop]
\node[draw,circle,scale=0.8,inner sep=0pt,fill=white] (A2) at (0.5,0.5) {$t_2$};
\node[draw,circle,scale=0.8,inner sep=0pt,fill=white] (A1) at (1.5,0.5) {$t_1$};
\node[draw,circle,scale=0.8,inner sep=0pt,fill=white] (A3) at (2.5,1.5) {$t_3$};
\node[draw,circle,scale=0.8,inner sep=0pt,fill=white] (A4) at (1.5,1.5) {$t_4$};
\node[draw,circle,scale=0.8,inner sep=0pt,fill=white] (A6) at (2.5,2.5) {$t_6$};
\node[draw,circle,scale=0.8,inner sep=0pt,fill=white] (A8) at (0.5,2.5) {$t_8$};
\end{scope}

\node[scale=\Fscl,blue,anchor=north west,inner sep=\Fsep] (F1) at (A1.south east) {$F_1$};
\node[scale=\Fscl,blue,anchor=north west,inner sep=\Fsep] (F2) at (A2.south east) {$F_2$};
\node[scale=\Fscl,blue,anchor=north west,inner sep=\Fsep] (F3) at (A3.south east) {$F_3$};
\node[scale=\Fscl,blue,anchor=north west,inner sep=\Fsep] (F4) at (A4.south east) {$F_4$};
\node[scale=\Fscl,blue,anchor=north west,inner sep=\Fsep] (F6) at (A6.south east) {$F_6$};
\node[scale=\Fscl,blue,anchor=north west,inner sep=\Fsep] (F8) at (A8.south east) {$F_8$};
\def\qlw{1pt}
\def\qop{1}
\draw[->,blue,opacity=\qop,line width=\qlw] (F2)--($(F2)+(0,\arrlen)$);
\draw[->,blue,opacity=\qop,line width=\qlw] (F1)--($(F1)+(0,\arrlen)$);
\draw[->,blue,opacity=\qop,line width=\qlw] (F1)--($(F1)+(-\arrlen,0)$);
\draw[->,blue,opacity=\qop,line width=\qlw] ($(F1)+(-\arrlen,\arrlen)$)--(F1);
\draw[->,blue,opacity=\qop,line width=\qlw] (F4)--($(F4)+(-\arrlen,\arrlen)$);
\draw[->,blue,opacity=\qop,line width=\qlw] (F3)--($(F3)+(0,\arrlen)$);
\draw[->,blue,opacity=\qop,line width=\qlw] (F3)--($(F3)+(-\arrlen,0)$);
\draw[->,blue,opacity=\qop,line width=\qlw] ($(F3)+(-\arrlen,\arrlen)$)--(F3);
\draw[->,blue,opacity=\qop,line width=\qlw] (F6)--($(F6)+(-\arrlen,0)$);
\end{tikzpicture}
};

\node(X) at (11,0){
\begin{tikzpicture}[scale=\scl]
\def\nodescl{0.7}

\node[circle,scale=0.8,inner sep=0pt,fill=white] (A2) at (0.5,0.5) {$\strut$};
\node[circle,scale=0.8,inner sep=0pt,fill=white] (A2) at (0.5,0.5) {$\strut$};
\node[circle,scale=0.8,inner sep=0pt,fill=white] (A1) at (1.5,0.5) {$\strut$};
\node[circle,scale=0.8,inner sep=0pt,fill=white] (A3) at (2.5,1.5) {$\strut$};
\node[circle,scale=0.8,inner sep=0pt,fill=white] (A4) at (1.5,1.5) {$\strut$};
\node[circle,scale=0.8,inner sep=0pt,fill=white] (A6) at (2.5,2.5) {$\strut$};
\node[circle,scale=0.8,inner sep=0pt,fill=white] (A8) at (0.5,2.5) {$\strut$};

\node[scale=\Fscl,blue,anchor=north west,inner sep=\Fsep] (F1) at (A1.south east) {$F_1$};
\node[scale=\Fscl,blue,anchor=north west,inner sep=\Fsep] (F2) at (A2.south east) {$F_2$};
\node[scale=\Fscl,blue,anchor=north west,inner sep=\Fsep] (F3) at (A3.south east) {$F_3$};
\node[scale=\Fscl,blue,anchor=north west,inner sep=\Fsep] (F4) at (A4.south east) {$F_4$};
\node[scale=\Fscl,blue,anchor=north west,inner sep=\Fsep] (F6) at (A6.south east) {$F_6$};
\node[scale=\Fscl,blue,anchor=north west,inner sep=\Fsep] (F8) at (A8.south east) {$F_8$};
\def\qlw{1pt}
\def\qop{1}
\draw[->,blue,opacity=\qop,line width=\qlw] (F2)--(F8);
\draw[->,blue,opacity=\qop,line width=\qlw,dashed] (F1)--(F4);
\draw[->,blue,opacity=\qop,line width=\qlw,dashed] (F1)--(F2);
\draw[->,blue,opacity=\qop,line width=\qlw] (F8)--(F1);
\draw[->,blue,opacity=\qop,line width=\qlw] (F4)--(F8);
\draw[->,blue,opacity=\qop,line width=\qlw,dashed] (F3)--(F4);
\draw[->,blue,opacity=\qop,line width=\qlw,dashed] (F3)--(F6);
\draw[->,blue,opacity=\qop,line width=\qlw] (F8)--(F3);
\draw[->,blue,opacity=\qop,line width=\qlw] (F6)--(F8);
\end{tikzpicture}
};

\end{tikzpicture}
}
  \caption{\label{fig:Le_quiver} Constructing a quiver $\QD$ from a Le-diagram $D$. See \cref{ex:Le_big}.}
\end{figure}
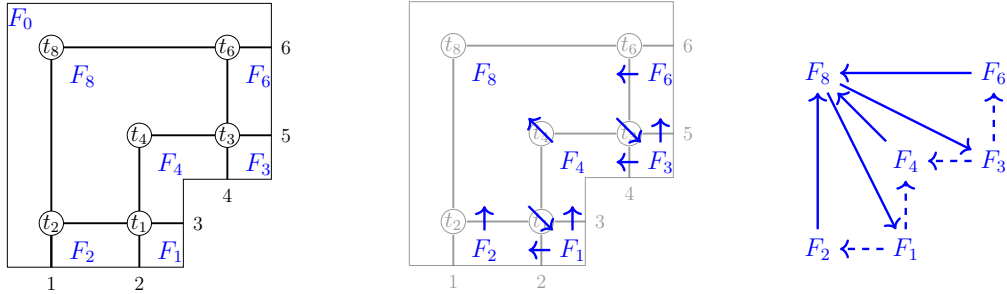
\end{example}

\subsection{Cluster algebra}\label{sec:cluster-algebra}
Let $Q$ be an ice quiver with vertex set $V$ partitioned as $V=V_f\sqcup V_m$, where $V_f$ (resp., $V_m$) denotes the set of frozen (resp., mutable) vertices. For each vertex $r\in  V$, introduce a formal variable $x_r$, and let $\xbf:=\{x_r\}_{r\in V}$. The \emph{cluster algebra} $\Acal(Q)$ associated to $Q$ is a certain $\C$-subalgebra of the ring $\C(\xbf)$ of rational functions in the variables $\xbf$. Explicitly, $\Acal(Q)$ is the subalgebra generated (as a ring) by all \emph{cluster variables} together with $\{x_r^{-1}\mid r\in V_f\}$, where the set of (in general infinitely many) cluster variables is constructed from the data $(Q,\xbf)$ using combinatorial operations called \emph{mutations}. Given a mutable vertex $r\in V_m$, a mutation at $r$ changes the quiver $Q$ in a certain way, and also replaces $x_r$ with
\begin{equation}\label{eq:x_mut}
x'_r:=\frac{\prod_{i\to r} x_{i}+\prod_{r\to j}x_{j}}{x_r},
\end{equation}
where the products are taken over arrows in $Q$ incident to $r$. We refer the reader to~\cite{FZ} for further background on cluster algebras.

\def\muu{\mu}

\section{Leclerc's cluster algebra}\label{sec:Leclerc_cluster_algebra}
For any $v\leq w$, Leclerc~\cite{Lec} introduced another cluster algebra using representations of preprojective algebras. In this section, we recast his construction in elementary terms when $(v,w)\in Q^J$ and show that in this case, his cluster algebra coincides with the one from \cref{sec:Le_cluster}.  The calculations in this section are very similar to those in~\cite[Section~5]{SSBW}, to which we refer the reader for an accessible introduction to preprojective algebra representations in type $A$.

\def\DQ{A_{n-1}}
\def\phir{\vec{\psi}}
\newcommand{\cev}[1]{\reflectbox{\ensuremath{\vec{\reflectbox{\ensuremath{#1}}}}}}
\def\phil{\cev{\psi}}

\def\cont{\operatorname{c}}
\subsection{Preprojective algebra representations from Young diagrams}\label{sec:quiver_rep}
Let $\DQ$ be the quiver with vertex set $[n-1]$ and a pair of opposite arrows between $i$ and $i+1$ for all $i\in[n-2]$. Recall that a \emph{representation} of $\DQ$ is a collection $E_1,\dots,E_{n-1}$ of vector spaces over $\C$ together with linear maps $\phir_{i}:E_i\to E_{i+1}$, $\phil_{i+1}:E_{i+1}\to E_i$ for all $i\in[n-2]$. A Young diagram $\la$ that fits inside a $k\times(n-k)$ rectangle gives rise to a representation of $\DQ$: for each box $(i,j)$ of $\la$, let $\cont(i,j):=k+j-i$ (thus $(i,j)$ is labeled by $s_{\cont(i,j)}$ in \cref{fig:Le}). Then for all $c\in[n-1]$, $E_{c}$ has a basis $\{e_{i,j}\mid (i,j)\in \la: \cont(i,j)=c\}$. Additionally, for each box $(i,j)$, the values of the maps $\phir_{\cont(i,j)}$ and $\phil_{\cont(i,j)}$ on $e_{i,j}$ are given by
\[\phir_{\cont(i,j)}(e_{i,j})=
  \begin{cases}
    e_{i,j+1}, &\text{if $(i,j+1)\in \la$,}\\
    0,&\text{otherwise;}
  \end{cases}\quad \phil_{\cont(i,j)}(e_{i,j})=
  \begin{cases}
    e_{i+1,j}, &\text{if $(i+1,j)\in \la$,}\\
    0,&\text{otherwise.}
  \end{cases} \]

Leclerc works not just with representations of $\DQ$, but with representations of the associated \emph{preprojective algebra} $\Lambda$. It is easy to see that each Young diagram $\la$ contained inside a $k\times (n-k)$ rectangle yields a representation $U_\la$ of $\Lambda$.

\def\lar#1{\nu_{\parr{#1}}}

\subsection{Leclerc's representations}\label{sec:Lec_rep}
Recall that we have fixed $(v,w)\in Q^J$. Leclerc associates a representation $U_r$ of $\Lambda$ to each $r\in \Jo$. Our goal is to define a family $\{\lar r\}_{r\in\Jo}$ of Young diagrams such that each of them fits inside a $k\times (n-k)$ rectangle, and such that $U_r=U_{\lar r}$ for all $r\in\Jo$.

\def\sv_#1{s^{\bv}_{#1}}
 For $r\in[m]$, we set 
\[\wi r:=s_{i_1}\cdots s_{i_r},\qquad\vi r:=\sv_{i_1}\cdots\sv_{i_r},\quad\text{where}\quad\sv_{i_r}:=
\begin{cases}
  s_{i_r}, &\text{if $r\notin\Jo$,}\\
  1,&\text{if $r\in \Jo$;}
\end{cases}\]
\[
\Wi r:=w^{-1}\cdot \wi r=s_{i_m}\cdots s_{i_{r+1}},\qquad \text{and}\qquad \Vi r:=v^{-1}\cdot \vi r=\sv_{i_m}\cdots\sv_{i_{r+1}}.
\]

For $a\in[n-1]$, let $\omega_{a}:= \{1,2,\ldots,a\}$. For $u \in W$ and $a \in [n-1]$, the subset $u\omega_a$ can be identified with a Young diagram $\muu(u,a)$ fitting inside an $(n-a) \times a$ rectangle, such that if one places $s_{a+i-j}$ inside each box $(i,j)\in\muu(u,a)$ and takes the product as in \cref{sec:QJtoLe}, the resulting element $\bar u$ satisfies $\bar u\omega_a=u\omega_a$. That is, $\bar u$ is the unique element of $W^{J_a}$ satisfying $\bar u\omega_a=u\omega_a$, where $J_a=[n-1]\setminus\{a\}$, and $\muu(u,a)$ is the Young diagram associated to $\bar u$. Clearly, if $u' \leq u$, then $\muu(u',a) \subseteq \muu(u,a)$.

Since $w\in W^J$, we see that $(\Wi{r-1})^{-1} \in W^J$ so $\muu(\Wi{r-1},i_r)$ is a rectangle for any $r$ whose top left (resp., bottom right) box is labeled by $s_a$ (resp., by $s_k$). Thus the $180^\circ$ rotation of the skew shape $\muu(\Wi{r-1},i_r)/\muu(\Vi{r-1},i_r)$ is a Young diagram which we denote $\lar r$.  We emphasize that $\lar r$ is defined for all $r\in[m]$. 

\begin{example}\label{ex:lar}
Let $k=6$, $n=12$. Consider a Le-diagram in \cref{fig:muu}(left). We have $\Jo=\{a,b,c,d\}$, and the diagrams $\lar r$ for $r\in\Jo$ are shown in \cref{fig:muu}(right). For instance, $i_b=7$, $\Wi{b-1}\omega_{i_b}=s_6s_7s_8s_9s_5s_6s_7s_8s_4s_5s_6s_7\omega_7$, and thus $\muu(\Wi{b-1},i_b)$ is a $3\times 4$ rectangle. Similarly, $\Vi{b-1}=s_7s_8s_9s_5s_6s_7s_8s_4s_6\omega_7=s_7s_8s_5s_6s_7\omega_7$. The shape $\muu(\Vi{b-1},i_b)=(3,2)$, rotated by $180^\circ$, consists of red and yellow squares (see \cref{fig:muu}) labeled by $s_7,s_8,s_5,s_6,s_7$, and is the complement of $\lar b$ inside a $3\times 4$ rectangle. Thus $\lar b=(4,2,1)$.
\end{example}

The next result follows from the definitions; see~\cite[Proposition~4.3]{Lec} and~\cite[Section~5]{SSBW}.
\begin{proposition}
For $r\in\Jo$, Leclerc's representation $U_r$ of $\Lambda$ coincides with the representation $U_{\lar r}$ constructed from $\lar r$ in \cref{sec:quiver_rep}.\qed
\end{proposition}

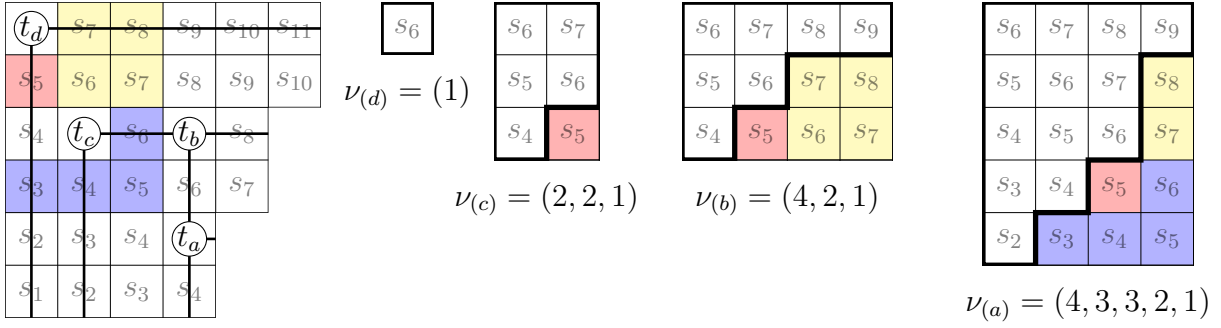
\begin{figure}
\def\scl{0.7}
\begin{tikzpicture}
\def\maxx{10}
\def\lineop{1}
\def\gridlw{0.2pt}
\def\tscl{1}
\def\lelw{1pt}
\def\col{red!30}
\def\coll{yellow!30}
\def\colll{blue!30}
\def\ydlw{2pt}

\def\sscl{1}
\def\sop{0.5}

\node[anchor=north west] (A) at (0,0){
\begin{tikzpicture}[xscale=\scl,yscale=-\scl,anchor=center]
\begin{scope}
    \clip (0,0)--(6,0)--(6,2)--(5,2)--(5,4)--(4,4)--(4,6)--(0,6)--cycle;
\fill[\col] (0,2) rectangle (1,1);
\fill[\coll] (1,2) rectangle (3,0);
\fill[\colll] (0,4) rectangle (3,3);
\fill[\colll] (2,3) rectangle (3,2);

\foreach[count=\ii] \i in {0,...,\maxx}{
\draw[line width=\gridlw] (0,\i) -- (\maxx,\i);
\draw[line width=\gridlw] (\i,0) -- (\i,\maxx);
}
\draw[line width=\gridlw] (0,0)--(6,0)--(6,2)--(5,2)--(5,4)--(4,4)--(4,6)--(0,6)--cycle;

\draw[line width=\lelw] (0.5,0.5)--(0.5,6);
\draw[line width=\lelw] (0.5,0.5)--(6,0.5);
\draw[line width=\lelw] (1.5,2.5)--(1.5,6);
\draw[line width=\lelw] (1.5,2.5)--(6,2.5);
\draw[line width=\lelw] (3.5,2.5)--(3.5,6);
\draw[line width=\lelw] (3.5,4.5)--(6,4.5);
\node[draw,circle,scale=\tscl,inner sep=0pt,fill=white] (d) at (0.5,0.5) {$t_d$};
\node[draw,circle,scale=\tscl,inner sep=0pt,fill=white] (c) at (1.5,2.5) {$t_c$};
\node[draw,circle,scale=\tscl,inner sep=0pt,fill=white] (b) at (3.5,2.5) {$t_b$};
\node[draw,circle,scale=\tscl,inner sep=0pt,fill=white] (a) at (3.5,4.5) {$t_a$};

\node[scale=\sscl,opacity=\sop] (A) at (1.5,0.5) {$s_7$};
\node[scale=\sscl,opacity=\sop] (A) at (0.5,1.5) {$s_5$};
\node[scale=\sscl,opacity=\sop] (A) at (1.5,1.5) {$s_6$};
\node[scale=\sscl,opacity=\sop] (A) at (0.5,2.5) {$s_4$};

\node[scale=\sscl,opacity=\sop] (A) at (2.5,0.5) {$s_8$};
\node[scale=\sscl,opacity=\sop] (A) at (3.5,0.5) {$s_9$};
\node[scale=\sscl,opacity=\sop] (A) at (2.5,1.5) {$s_7$};
\node[scale=\sscl,opacity=\sop] (A) at (3.5,1.5) {$s_8$};
\node[scale=\sscl,opacity=\sop] (A) at (2.5,2.5) {$s_6$};

\node[scale=\sscl,opacity=\sop] (A) at (0.5,3.5) {$s_3$};
\node[scale=\sscl,opacity=\sop] (A) at (1.5,3.5) {$s_4$};
\node[scale=\sscl,opacity=\sop] (A) at (2.5,3.5) {$s_5$};
\node[scale=\sscl,opacity=\sop] (A) at (3.5,3.5) {$s_6$};
\node[scale=\sscl,opacity=\sop] (A) at (0.5,4.5) {$s_2$};
\node[scale=\sscl,opacity=\sop] (A) at (1.5,4.5) {$s_3$};
\node[scale=\sscl,opacity=\sop] (A) at (2.5,4.5) {$s_4$};
\node[scale=\sscl,opacity=\sop] (A) at (0.5,5.5) {$s_1$};
\node[scale=\sscl,opacity=\sop] (A) at (1.5,5.5) {$s_2$};
\node[scale=\sscl,opacity=\sop] (A) at (2.5,5.5) {$s_3$};
\node[scale=\sscl,opacity=\sop] (A) at (3.5,5.5) {$s_4$};

\node[scale=\sscl,opacity=\sop] (A) at (4.5,0.5) {$s_{10}$};
\node[scale=\sscl,opacity=\sop] (A) at (4.5,1.5) {$s_{9}$};
\node[scale=\sscl,opacity=\sop] (A) at (4.5,2.5) {$s_{8}$};
\node[scale=\sscl,opacity=\sop] (A) at (4.5,3.5) {$s_{7}$};

\node[scale=\sscl,opacity=\sop] (A) at (5.5,0.5) {$s_{11}$};
\node[scale=\sscl,opacity=\sop] (A) at (5.5,1.5) {$s_{10}$};

\end{scope}
\end{tikzpicture}
};

\node[anchor=north west] (dd) at (5,0){
\begin{tikzpicture}[scale=\scl,anchor=center]
\begin{scope}
    \clip (0,0)--(1,0)--(1,1)--(0,1)--cycle;
\foreach[count=\ii] \i in {0,...,\maxx}{
\draw[line width=\gridlw] (0,\i) -- (\maxx,\i);
\draw[line width=\gridlw] (\i,0) -- (\i,\maxx);
}
\draw[line width=\ydlw] (0,0)--(1,0)--(1,1)--(0,1)--cycle;
\node[scale=\sscl,opacity=\sop] (A) at (0.5,0.5) {$s_6$};
\end{scope}
\end{tikzpicture}
};
\node[anchor=north] (ddd) at (dd.south){$\lar{d}=(1)$};

\node[anchor=north west] (cc) at (6.5,0){
\begin{tikzpicture}[xscale=\scl,yscale=-\scl,anchor=center]
\begin{scope}
    \clip (0,0)--(2,0)--(2,3)--(0,3)--cycle;

\fill[\col] (1,3) rectangle (2,2);
\foreach[count=\ii] \i in {0,...,\maxx}{
\draw[line width=\gridlw] (0,\i) -- (\maxx,\i);
\draw[line width=\gridlw] (\i,0) -- (\i,\maxx);
}
\draw[line width=\gridlw](0,0)--(2,0)--(2,3)--(0,3)--cycle;
\draw[line width=\ydlw](0,0)--(2,0)--(2,2)--(1,2)--(1,3)--(0,3)--cycle;

\node[scale=\sscl,opacity=\sop] (A) at (0.5,0.5) {$s_6$};
\node[scale=\sscl,opacity=\sop] (A) at (1.5,0.5) {$s_7$};
\node[scale=\sscl,opacity=\sop] (A) at (0.5,1.5) {$s_5$};
\node[scale=\sscl,opacity=\sop] (A) at (1.5,1.5) {$s_6$};
\node[scale=\sscl,opacity=\sop] (A) at (0.5,2.5) {$s_4$};
\node[scale=\sscl,opacity=\sop] (A) at (1.5,2.5) {$s_5$};
\end{scope}
\end{tikzpicture}
};
\node[anchor=north] (ccc) at (cc.south){$\lar{c}=(2,2,1)$};

\node[anchor=north west] (bb) at (9,0){
\begin{tikzpicture}[xscale=\scl,yscale=-\scl,anchor=center]
\begin{scope}
    \clip (0,0) rectangle (4,3);

\fill[\col] (1,3) rectangle (2,2);
\fill[\coll] (2,3) rectangle (4,1);

\foreach[count=\ii] \i in {0,...,\maxx}{
\draw[line width=\gridlw] (0,\i) -- (\maxx,\i);
\draw[line width=\gridlw] (\i,0) -- (\i,\maxx);
}
\draw[line width=\gridlw](0,0) rectangle (4,3);
\draw[line width=\ydlw](0,0)--(4,0)--(4,1)--(2,1)--(2,2)--(1,2)--(1,3)--(0,3)--cycle;

\node[scale=\sscl,opacity=\sop] (A) at (0.5,0.5) {$s_6$};
\node[scale=\sscl,opacity=\sop] (A) at (1.5,0.5) {$s_7$};
\node[scale=\sscl,opacity=\sop] (A) at (0.5,1.5) {$s_5$};
\node[scale=\sscl,opacity=\sop] (A) at (1.5,1.5) {$s_6$};
\node[scale=\sscl,opacity=\sop] (A) at (0.5,2.5) {$s_4$};
\node[scale=\sscl,opacity=\sop] (A) at (1.5,2.5) {$s_5$};

\node[scale=\sscl,opacity=\sop] (A) at (2.5,0.5) {$s_8$};
\node[scale=\sscl,opacity=\sop] (A) at (3.5,0.5) {$s_9$};
\node[scale=\sscl,opacity=\sop] (A) at (2.5,1.5) {$s_7$};
\node[scale=\sscl,opacity=\sop] (A) at (3.5,1.5) {$s_8$};
\node[scale=\sscl,opacity=\sop] (A) at (2.5,2.5) {$s_6$};
\node[scale=\sscl,opacity=\sop] (A) at (3.5,2.5) {$s_7$};
\end{scope}
\end{tikzpicture}
};
\node[anchor=north] (bbb) at (bb.south){$\lar{b}=(4,2,1)$};

\node[anchor=north west] (aa) at (13,0){
\begin{tikzpicture}[xscale=\scl,yscale=-\scl,anchor=center]
\begin{scope}
    \clip (0,0) rectangle (4,5);

\fill[\col] (2,4) rectangle (3,3);
\fill[\coll] (3,3) rectangle (4,1);
\fill[\colll] (1,5) rectangle (4,4);
\fill[\colll] (3,4) rectangle (4,3);

\foreach[count=\ii] \i in {0,...,\maxx}{
\draw[line width=\gridlw] (0,\i) -- (\maxx,\i);
\draw[line width=\gridlw] (\i,0) -- (\i,\maxx);
}
\draw[line width=\gridlw] (0,0) rectangle (4,5);
\draw[line width=\ydlw](0,0)--(4,0)--(4,1)--(3,1)--(3,3)--(2,3)--(2,4)--(1,4)--(1,5)--(0,5)--cycle;

\node[scale=\sscl,opacity=\sop] (A) at (0.5,0.5) {$s_6$};
\node[scale=\sscl,opacity=\sop] (A) at (1.5,0.5) {$s_7$};
\node[scale=\sscl,opacity=\sop] (A) at (0.5,1.5) {$s_5$};
\node[scale=\sscl,opacity=\sop] (A) at (1.5,1.5) {$s_6$};
\node[scale=\sscl,opacity=\sop] (A) at (0.5,2.5) {$s_4$};
\node[scale=\sscl,opacity=\sop] (A) at (1.5,2.5) {$s_5$};

\node[scale=\sscl,opacity=\sop] (A) at (2.5,0.5) {$s_8$};
\node[scale=\sscl,opacity=\sop] (A) at (3.5,0.5) {$s_9$};
\node[scale=\sscl,opacity=\sop] (A) at (2.5,1.5) {$s_7$};
\node[scale=\sscl,opacity=\sop] (A) at (3.5,1.5) {$s_8$};
\node[scale=\sscl,opacity=\sop] (A) at (2.5,2.5) {$s_6$};
\node[scale=\sscl,opacity=\sop] (A) at (3.5,2.5) {$s_7$};

\node[scale=\sscl,opacity=\sop] (A) at (0.5,3.5) {$s_3$};
\node[scale=\sscl,opacity=\sop] (A) at (1.5,3.5) {$s_4$};
\node[scale=\sscl,opacity=\sop] (A) at (2.5,3.5) {$s_5$};
\node[scale=\sscl,opacity=\sop] (A) at (3.5,3.5) {$s_6$};
\node[scale=\sscl,opacity=\sop] (A) at (0.5,4.5) {$s_2$};
\node[scale=\sscl,opacity=\sop] (A) at (1.5,4.5) {$s_3$};
\node[scale=\sscl,opacity=\sop] (A) at (2.5,4.5) {$s_4$};
\node[scale=\sscl,opacity=\sop] (A) at (3.5,4.5) {$s_5$};

\end{scope}
\end{tikzpicture}
};
\node[anchor=north] (aaa) at (aa.south){$\lar{a}=(4,3,3,2,1)$};

\end{tikzpicture}
  \caption{\label{fig:muu}Constructing Young diagrams $\lar r$ from a Le-diagram (cf. \cref{ex:lar}).}
\end{figure}

\def\ir{i^{\parr r}}
\def\jr{j^{\parr r}}
\def\Hook{\operatorname{H}}
We give an alternative description of the Young diagram $\lar r$ using the combinatorics of Le-diagrams. 

For integers $a,b\geq1$, denote by $\Hook(a,b)=(a,1^{b-1})$ the hook Young diagram whose first row contains $a$ boxes and whose first column contains $b$ boxes. We consider ``Frobenius coordinates'' for Young diagrams: we write 
$\mu=[(a_1,b_1),\dots,(a_d,b_d)]$ if   $\mu=\{(l+i-1,l+j-1)\mid l\in[d], i\in[a_l], j\in [b_l]\}$. 
Thus the first row of $\mu$ has $a_1$ boxes, the first column of $\mu$ has $b_1$ boxes, etc.  

For each box $(i,j)\in \la$, let $\NW(i,j)\in\la\sqcup\{(0,0)\}$ be the box closest to $(i,j)$ in the strictly northwest direction that is either $(0,0)$ or contains a dot. Observe that $\NW(i,j)$ is well defined because of the Le-diagram condition in \cref{sec:QJtoLe}. 

Recall from \cref{sec:QJtoLe} that the boxes of $\la$ correspond to the terms in the reduced word $w=s_{i_1}\cdots s_{i_m}$ for $w$, i.e., to the elements of $[m]$. For $r\in[m]$, we denote by $(\ir,\jr)\in \la$ the corresponding box of $\la$ (thus $i_r=k+\jr-\ir$). 

\def\irW{i^{\parr{\rW}}}
\def\jrW{j^{\parr{\rW}}}
\def\irN{i^{\parr{\rN}}}
\def\jrN{j^{\parr{\rN}}}
\def\irNW{i^{\parr{\rNW}}}
\def\jrNW{j^{\parr{\rNW}}}

\def\irp{i^{\parr{r'}}}
\def\jrp{j^{\parr{r'}}}

\begin{proposition}\label{prop:Lecdiagram}
For $r\in \Jo$, the Frobenius coordinates of $\lar r$ are given by 
\[
  \lar r = \left[(\ir,\jr),\NW(\ir,\jr),\ldots,\NW^p(\ir,\jr)\right],\]
where $p\geq0$ is such that $\NW^{p+1}(\ir,\jr) = (0,0)$.
\end{proposition}
\begin{proof}
For a box $(i,j)\in \la$, let $\WNW(i,j)$ be the box closest to $(i,j)$ in the \emph{weakly} northwest direction that is either $(0,0)$ or contains a dot. Thus $(i,j)$ contains a dot if and only if $\WNW(i,j)=(i,j)$. Recall that $\lar r$ is defined for all $r\in[m]$, not just for $r\in\Jo$. We will show more generally that for $r\in[m]$, we have 
\begin{equation*}%
  \lar r = \left[\WNW(\ir,\jr),\NW(\WNW(\ir,\jr)),\ldots,\NW^p(\WNW(\ir,\jr))\right],
\end{equation*}
where $p\geq0$ satisfies $\NW^{p+1}(\WNW(\ir,\jr)) = (0,0)$. If $(\ir,\jr)=(1,1)$ then the result is clear. Otherwise, assume that we have shown the result for all boxes weakly northwest of $(\ir,\jr)$ other than $(\ir,\jr)$. 

Suppose first that $r\notin\Jo$, thus $(\ir,\jr)$ does not contain a dot. Let $\rW\in[m]$ and $\rN\in[m]$ be such that $(\irW,\jrW)=(\ir,\jr-1)$ and $(\irN,\jrN)=(\ir-1,\jr)$. By the Le-diagram condition, either all boxes above or all boxes to the left of $(\ir,\jr)$ do not contain a dot. In the former case, we have $\lar r= \lar {\rW}$ and $\WNW(\ir,\jr)=\WNW(\irW,\jrW)$, and in the latter case, we have $\lar r= \lar {\rN}$ and $\WNW(\ir,\jr)=\WNW(\irN,\jrN)$. The result follows by induction.

Suppose now that $r\in\Jo$ and let $\rNW\in[m]$ be such that $(\irNW,\jrNW)=(\ir-1,\jr-1)$. If $r\notin\Jo$, then clearly $\muu(\Vi{r-1},i_r)$ is a disjoint union of $\Hook(\ir,\jr)$ and $\muu(\Vi{\rNW-1},i_{\rNW})$, so $\lar r=\lar{\rNW}$. If $r\in\Jo$, then $\muu(\Vi{r-1},i_r)=\muu(\Vi{\rNW-1},i_{\rNW})$, so $\lar r$ is a disjoint union of $\Hook(\ir,\jr)$ and $\lar{\rNW}$.
\end{proof}

\def\ip#1{i^{(#1)}}
\def\jp#1{j^{(#1)}}
\begin{example}
Continuing \cref{ex:lar}, the Frobenius coordinates of $\lar r$ for $r\in\Jo$ are given as follows:
\[\lar d=[(1,1)],\quad \lar c=[(2,3),(1,1)],\quad \lar b=[(4,3),(1,1)],\quad \lar a=[(4,5),(2,3),(1,1)].\]
We see that $\lar a$ is a union of $\Hook(\ip a,\jp a)=\Hook(4,5)$ and $\lar c$, which is a union of $\Hook(\ip c,\jp c)=\Hook(2,3)$ and $\lar d=\Hook(1,1)$. Similarly, $\lar b$ is a union of $\Hook(\ip b,\jp b)=\Hook(4,3)$ and $\lar d$.
\end{example}

\subsection{Leclerc's quiver}

Now that we have constructed Young diagrams $\lar r$ for $r\in\Jo$, we can analyze the quiver $\tQ$ that Leclerc associates to $(v,w)\in Q^J$. The vertex set\footnote{More precisely, the vertex set of $\tQ$ consists of irreducible factors of functions $\{f_r\}_{r\in\Jo}$ defined in~\eqref{eq:f_r}. However,  when $(v,w)\in Q^J$, each $f_r$ is irreducible by \cref{cor:irred} below, thus we may label the vertices of $\tQ$ by elements of $\Jo$.}  of $\tQ$ is just $\Jo$. The frozen vertices of $\tQ$ correspond to the Young diagrams obtained from $\muu(w^{-1},i_a)/\muu(v^{-1},i_a)$ (for $a\in[n-1]$) by a $180^\circ$ rotation. It is easy to see that these are precisely the Young diagrams $\lar r$ such that $F_r$ is a boundary face of $G(D)$ that contains the part of the boundary of $\la$ between boundary vertices $a$ and $a+1$. Thus the map $r\mapsto F_r$ sends the vertices of $\tQ$ bijectively to the vertices of the quiver $\QD$ from \cref{Le_quiver}, preserving the partition into frozen and mutable vertices.

The arrows of $\tQ$ can be described in terms of morphisms of Young diagrams.  Given a skew shape $\la/\mu$ for  $\mu\subset \la$, we say that their set-theoretic difference $\la/\mu$ is an \emph{order ideal of $\la$}. For a Young diagram $\la$ and an integer $p\geq0$, we denote $\shift^p(\la):=\{(i+p,j+p)\mid (i,j)\in\la\}$. For another Young diagram $\mu$, we write $\la\xto p \mu$ if the set $\shift^p(\la)\cap \mu$ is an order ideal of $\mu$. In this case, we say that $\la\xto p\mu$ is a \emph{morphism from $\la$ to $\mu$}.  The morphism $\lambda \xto 0 \lambda$ is considered trivial, and a morphism $\la \xto p \mu$ is the zero morphism if $\shift^p(\la)\cap \mu = \emptyset$.  Morphisms can be composed: if $\lambda \xto p \mu$ and $\mu \xto q \nu$ then $\lambda \xto{p+q} \nu$.  For $r,r'\in \Jo$, a nonzero and nontrivial morphism $\lar r\xto p\lar{r'}$ is \emph{irreducible} if it is not a composition of non-trivial morphisms $\lar r\xto{p''} \lar{r''}\xto{p'}\lar{r'}$ for some $r''\in\Jo$. Leclerc's quiver $\tQ$ contains an arrow $r \to r'$ for $r \neq r'\in\Jo$ if and only if at least one of $r,r'$ is mutable and there is an irreducible morphism $\lar r \xto p \lar{r'}$ for some $p\geq0$.

\begin{remark}
We explain the relation between our morphisms of Young diagrams and the morphisms of representations of $\Lambda$ from the original definition~\cite{Lec} of $\tQ$.
Let $U_{\la}$ and $U_{\mu}$ be the two (indecomposable) $\Lambda$-modules associated to Young diagrams $\la$ and $\mu$ as in Section~\ref{sec:quiver_rep}.    Since $U_{\la}$ is generated (as a $\Lambda$-module) by the vector $e_{1,1}$, a morphism $f:U_{\la} \to U_{\mu}$ is uniquely determined by $f(e_{1,1}) \in U_{\mu}$, which must be a linear combination of $e_{1,1}, e_{2,2}, \ldots, e_{d,d} \in U_{\mu}$, where $d$ is the length of $\mu$ in Frobenius coordinates.

Associated to each morphism $\la \xto p \mu$ of Young diagrams is the \emph{elementary morphism} $U_\la \xto p U_\mu$ of $\Lambda$-modules sending $e_{1,1}$ to $e_{p,p}$.  (The condition that $\shift^p(\la) \cap \mu$ is an order ideal of $\mu$ corresponds exactly to the condition that $e_{1,1} \mapsto e_{p,p}$ defines a morphism of $\Lambda$-modules; see~\cite[Remark~5.15]{SSBW}.)  Any morphism $f: U_\la \to U_\mu$ is thus a linear combination of the elementary morphisms $U_\la \xto p U_\mu$ of Young diagrams.

A morphism $f: U_{\lar r} \to U_{\lar{r'}}$ is \emph{irreducible} if it is nonzero, not an isomorphism, and cannot be factored nontrivially within the category ${\rm add}(U)$ whose objects are isomorphic to direct sums of the $U_{\lar s}$ for $s \in \Jo$.  Leclerc's quiver $\tQ$ has no loops, i.e., arrows from $r$ to $r$; see~\cite[Definition~3.9(d)]{Lec}.  For $r \neq r'$, the number of arrows in $\tQ$ from $r$ to $r'$ is equal to the \emph{dimension of the space of irreducible morphisms} from $U_{\lar r}$ to $U_{\lar{r'}}$.  This dimension is defined \cite{BIRS,Schiffler} to be equal to the dimension of the space of all morphisms $U_{\lar r} \to U_{\lar{r'}}$ modulo the subspace consisting of reducible morphisms. If $\la \xto p \mu$ is a morphism, then (see the proof of Proposition \ref{prop:Lecagree}) $\la \xto q \mu$ is a reducible morphism for all $q > p$.  Thus any morphism $f: U_{\lar r} \to U_{\lar{r'}}$ is equal modulo reducible morphisms to a scalar multiple of some elementary morphism $U_\la \xto p U_\mu$.  It follows that the dimension of the space of irreducible morphisms from $U_{\lar r}$ to $U_{\lar{r'}}$ is equal to 1 or 0, depending on whether there is an elementary morphism $U_{\lar r} \xto p U_{\lar{r'}}$ that is irreducible as a morphism of $\Lambda$-modules.  Furthermore, when considering the irreducibility of an elementary morphism $U_{\lar r} \xto p U_{\lar{r'}}$, we only need to check if it factors nontrivially as a product of elementary morphisms.  Thus the irreducibility of $U_{\lar r} \xto p U_{\lar{r'}}$ in the sense of $\Lambda$-modules agrees with the notion of irreducibility we defined for a morphism $\lar r \xto p \lar{r'}$ of Young diagrams.
\end{remark}

\begin{proposition}\label{prop:Lecagree}
The map $r\mapsto F_r$ gives a quiver isomorphism between Leclerc's quiver $\tQ$ and the Le-diagram quiver $\QD$.
\end{proposition}
\begin{proof}
If the Frobenius coordinates of a Young diagram $\la$ are given by $\la=[(a_1,b_1),\dots,(a_d,b_d)]$, we set $\la^1:=(a_1,b_1),\dots,\la^d:=(a_d,b_d)$, and $\la^{d+1}=\la^{d+2}=\dots=(0,0)$. Let us write $(a,b)\geq (a',b')$  if $a\geq a'$ and $b\geq b'$.

Let $\la \xto p \mu$ be a morphism.  Since its image is an order ideal of $\mu$, a morphism $\la\xto p\mu$ exists if and only if $\la^1\geq \mu^{p+1}$, $\la^2\geq \mu^{p+2}$, etc. Moreover, if $p<q$ and we have a morphism $\la\xto p\mu$ then  the morphism $\la\xto q\mu$ is not irreducible: it factors through $\la\xto p \mu\xto{q-p} \mu$. 
Also note that $\la\xto0\mu$ if and only if $\mu\subset \la$. We write $\la\xxto p \mu$ if the morphism $\la\xto p\mu$ exists and is \emph{injective} (i.e., $\shift^p(\la)\subset \mu$ is an order ideal of $\mu$). This is equivalent to $\la^1= \mu^{p+1}$, $\la^2= \mu^{p+2}$, etc.

By \cref{prop:Lecdiagram}, for each $r\in\Jo$, we have $\lar r^1=(\ir,\jr)$. Thus all Young diagrams $\{\lar r\}_{r\in\Jo}$ are different.  Observe that the morphism $\la\xxto 0\mu$ exists if and only if $\la=\mu$, in which case it is a trivial morphism. 

Let $r\in\Jo$ and suppose that the neighboring faces $F_a,F_b,F_c,F_r$ of $t_r$ in $G(D)$ are labeled as in \cref{fig:neigh}. It follows from \cref{prop:Lecdiagram} that we have morphisms $\lar r\xto 0 \lar a$, $\lar r\xto 0 \lar b$, $\lar c\xxto 1\lar r$.

We first show that for $r,r'\in\Jo$, if there is no arrow $F_r\to F_{r'}$ in $\QD$ then there is no arrow $r\to r'$ in $\tQ$. Indeed, let $\nu:=\lar r$ and $\nu':=\lar{r'}$, and  suppose that there is no arrow $F_r\to F_{r'}$ in $\QD$ but we have a morphism $\nu\xto p\nu'$ for some $p\geq0$. Assume that the regions around $t_{r'}$ in $G(D)$ are labeled by $F_{a'},F_{b'},F_{c'},F_{r'}$ as in \cref{fig:neigh}. 
If $p \geq 1$, then $\nu\xto p\nu'$ factors through $\nu \xto{p-1}\lar{c'}\xto{1}\nu'$. Therefore such a morphism is not irreducible unless $p=1$ and the morphism $\nu\xto{p-1}\lar{c'}$ is trivial, i.e., $r=c'$. Since there is no arrow $F_r\to F_{r'}$ in $\QD$, we must have either $c'=a'$, or $c'=b'$, or both. Without loss of generality, assume that $c'=a'(=r)$. If both $r$ and $r'$ are frozen then $\tQ$ contains no arrow between them. Thus at least one of them must be mutable, so the horizontal edge of $G(D)$ between $r$ and $r'$ must have another vertex to the right of $t_{r'}$. Let that vertex be labeled by $t_{q}$, then we have morphisms $\nu\xto 1\lar q\xto 0 \nu'$, thus we see that indeed our morphism $\nu\xto p\nu'$ is not irreducible when $p\geq 1$.

Assume now that $p = 0$, which implies that $(\ir,\jr) \geq (\irp,\jrp)$ and $(\ir,\jr) \neq (\irp,\jrp)$.  From the definition of the Le-diagram quiver $\QD$, we see that there exists a path from $F_r$ to $F_{r'}$ that consists of arrows all going up, left, or up-left. The composition of the corresponding morphisms gives the morphism $\nu\xto0\nu'$, which shows that it is not irreducible if there is no arrow $F_r\to F_{r'}$ in $\QD$.

It remains to show that if we have an arrow $F_r\to F_{r'}$ in $\QD$ then the morphism $\nu\to\nu'$ is irreducible. Let the regions around $t_r$ (resp., $t_{r'}$) be labeled by $F_a,F_b,F_c,F_r$ (resp., $F_{a'},F_{b'},F_{c'},F_{r'}$) as above. First, suppose that $r=c'$, in which case our morphism is $\nu\xxto1\nu'$.  
If it factors as $\nu \xto p \nu'' \xto q \nu'$ then the injectivity of $\nu \xxto1 \nu'$ forces $\nu \xto p \nu''$ to be injective, and thus we must have $\nu \xxto 1 \nu'' \xto 0 \nu'$.  But then $\nu$ is obtained from $\nu''$ by removing a hook, so by  \cref{prop:Lecdiagram}, $\nu \xxto 1 \nu''$ must be one of the down-right arrows in $\QD$. In this case, the hooks $(\nu'')^1$ and $(\nu')^1$ have to be incomparable, contradicting $\nu'' \xto0 \nu'$. 
Next, suppose that $r'$ equals to $a$ or $b$. Then we have a morphism $\nu\xto0 \nu'$. It is clear from the definition of $\QD$ that if $F_r\to F_{r'}$ is an arrow of $\QD$ then $\QD$ contains no directed path from $F_r$ to $F_{r'}$ of length more than $1$. If the morphism $\nu\xto0\nu'$ is not irreducible then there must be such a directed path from $r$ to $r'$ in $\tQ$. But we have already shown that each arrow of $\tQ$ appears as an arrow in $\QD$, thus $\nu\xto0\nu'$ must be irreducible.
\end{proof}

\section{Clusters and positroid varieties}\label{sec:positroid-varieties}

\Cref{prop:Lecagree} shows that the two (abstract) cluster algebras $\AQD$ and $\AQT$ are isomorphic. In this section, we further connect them by showing that the conjectural cluster structures they define on $\C[\PR_v^w]$ coincide.

\subsection{Background on positroid varieties}\label{sec:pos}
Let $G = \SL_n(\C)$ and $B,B_-,N,N_-$ denote the upper- and lower-triangular Borel subgroups, and their unipotent parts.  For $i\in [n-1]$, denote by $\ds_i\in G$ a (signed) permutation matrix representing $s_i\in W$ that has a $2\times 2$ block equal to $\begin{pmatrix}
0 & 1\\
-1 & 0
\end{pmatrix}$ in rows and columns $i,i+1$. Given a reduced word $w=s_{i_1}\cdots s_{i_m}$ for $w\in W$, we let $\dw \in G$ denote the (signed) permutation matrix given by $\dw:=\ds_{i_1}\cdots \ds_{i_m}$.  For $v \leq w$, we define the \emph{open Richardson variety} $\Rich_v^w$ to be the image of $B\dv B_- \cap B_- \dw B_-$ in $G/B_-$.  Thus $\Rich_v^w$ is a smooth affine subvariety of $G/B_-$.

Recall that $J = [n] \setminus \{k\}$. Let $\PJ \supset B_-$ denote the $J$-parabolic subgroup such that the projection $\pi_J: G \to G/\PJ \simeq \Gr(n-k,n)$ is given by sending a $n \times n$ matrix $g$ to the column span of its last $n-k$ columns. We sometimes denote $\pi_J(g)$ by $g\PJ$. For an $(n-k)$-element subset $I$ of $[n]$, we denote by $\Delta_I$ the corresponding \emph{Pl\"ucker coordinate} on $\Gr(n-k,n)$, i.e., the maximal $(n-k)\times (n-k)$ minor of $g$ with row set $I$ and column set $[k+1,n]$.

For $(v,w) \in Q^J$, the \emph{open positroid variety} is the image $\PR_v^w := \pi_J(\Rich_v^w) \subset \Gr(n-k,n)$; see \cite{Lus,KLS}.  It is isomorphic to $\Rich_v^w$, and is a smooth affine subvariety of $\Gr(n-k,n)$.  For other descriptions of $\PR_v^w$, see \cite{Pos,BGY}.

\subsection{Leclerc's functions}
Fix $v \leq w$.  Following~\cite[Section~2]{Lec}, denote $N'(v) = N\cap (\dv^{-1}N \dv) \subset N$ and  $N_{v,w}:=N'(v)\cap \dv^{-1}B_-\dw B_-$. We will be interested in the variety 

\begin{equation}\label{eq:vNvw}
\dv N_{v,w}= \dv N\cap N\dv\cap B_-\dw B_-.
\end{equation}

\begin{lemma}[{\cite[Theorem~2.3]{BGY}, \cite[Lemma~2.2]{Lec}}]\label{lemma:vNvw}
The map $\dv N_{v,w} \to G/B_-$ given by $g \mapsto gB_-$ gives an isomorphism $\dv N_{v,w}\xrasim \Rich_v^w$.
\end{lemma}

\begin{remark}\label{rmk:easy_hard}
Both~\cite{Lec} and~\cite{SSBW} work with the left-sided flag variety $B_-\bs G$. In particular, Leclerc shows that the map $g\mapsto B_-g$ gives an isomorphism $\dv N_{v,w}\xrasim B_-\bs (B_-\dv B\cap B_-\dw B_-)$. \Cref{lemma:vNvw} follows from this statement by replacing $g,v,w$ with their inverses. However, since in either case one works with \emph{rightmost} positive distinguished subexpressions, switching from $B_-\bs G$  to $G/B_-$ has a drastic effect on the combinatorics of Leclerc's quivers, as one can see by comparing \cref{sec:Leclerc_cluster_algebra} with~\cite[Section~7]{Lec} or~\cite{SSBW}. In fact, when working with $B_-\bs G$, Leclerc's cluster structure does not in general coincide with (either source or target labeled versions of) the cluster structure coming from Postnikov diagrams; see~\cite[Appendix~B]{SSBW}.
\end{remark}

For $u \in W$ and $a \in [n-1]$, let $\omega_a$ and $\muu(u,a)$ be as in \cref{sec:Lec_rep}. Leclerc~\cite{Lec} considers a family of functions on the unipotent group $N$: for each $r\in \Jo$, the corresponding function~is
\begin{equation}\label{eq:f_r}
f_r:=\Delta_{\Vi{r-1}\omega_{i_r} ,\Wi{r-1}\omega_{i_r} }:N\to \C
\end{equation}
where $\Delta_{A,B}$ is the minor whose rows and columns are indexed by $A$ and $B$ respectively.    The functions $f_r$ restrict to functions on $N_{v,w}$. Leclerc proves that the irreducible (as elements of $\C[N]$) factors of $\{f_r \mid r \in \Jo\}$ form the initial cluster variables of a cluster subalgebra of $\C[N_{v,w}]$.  

\begin{lemma}\label{lem:skewirred}
For $u'\leq u\in W$, if the skew shape $\muu(u,a)/\muu(u',a)$ is connected then the minor $\Delta_{u'\omega_a, u\omega_a}$ is an irreducible element of $\C[N]$.
\end{lemma}
\begin{proof}
The polynomial $\Delta_{u'\omega_a, u\omega_a}(g)$ is homogeneous with $\deg(g_{i,j})=j-i$. Restricting to the subspace of $N$ consisting of matrices constant along diagonals, we see that the result is implied by the Jacobi-Trudi formula combined with the irreducibility~\cite[Theorem~1]{RSvWcorr} of skew Schur functions indexed by connected skew shapes.
\end{proof}

Suppose now that $(v,w)\in Q^J$. As we have established in \cref{sec:Lec_rep}, for all $r\in \Jo$, the skew shape $\muu(\Wi{r-1},i_r)/\muu(\Vi{r-1},i_r)$ is a $180^\circ$ rotation of a Young diagram $\lar r$, thus we have shown the following.
\begin{corollary}\label{cor:irred}
For $(v,w)\in Q^J$ and all $r\in \Jo$, the function $f_r$ defined in~\eqref{eq:f_r} is an irreducible element of $\C[N]$.
\end{corollary}

\subsection{Face labels}\label{sec:face-labels}
So far the faces of $G(D)$ have been labeled by an abstract set $\{F_r\}_{r\in \Jo\sqcup \{0\}}$.  We now identify each face $F_r$ with an $(n-k)$-element subset of $[n]$, so that it would correspond to a Pl\"ucker coordinate on $\Gr(n-k,n)$.

\def\wirelw{1pt}
\def\dotscl{1}
\def\roc{5}
\def\setscl{1}
\def\rbound{19}
\def\lbound{-2.5}
\def\bnscl{1}

\def\leop{0.5}
\def\strandlw{1pt}

The graph $G(D)$ has $n$ distinguished paths $p_1,p_2,\ldots,p_n$ connecting boundary vertices, called \emph{strands}. For $a\in[n]$, the strand $p_a$ starts\footnote{This is called the \emph{source-labeling} of strands. For the other convention, called \emph{target-labeling}, the path $p_a$ \emph{ends} at vertex $a$.} at the boundary vertex labeled $a$, and then travels along the edges of $G(D)$, making turns at each vertex $t_r$ according to the following ``rules of the road'' (cf.~\cite[Figure~20.2]{Pos}):
\def\leop{0.4}
\begin{equation}\label{eq:rules_road}
\begin{tikzpicture}[baseline=(Z.base)]
\def\nodescale{0.8}
\def\qlw{1pt}
\def\qop{1}
\coordinate(Z) at (0,0);

\def\Finnersep{2pt}
\def\stp{3}

\def\cola{red}
\def\colb{blue}
\def\colc{green!70!black}
\def\cold{brown!80!black}
\def\roc{10}
\def\arrlw{1pt}
\def\one{1}

\node[scale=\nodescale](A) at (0,0){
\begin{tikzpicture}
\draw[line width=0.8pt,opacity=\leop] (0,0)--(1,0);
\draw[line width=0.8pt,opacity=\leop] (0,0)--(0,-1);
\node[draw=white,circle,scale=0.8,inner sep=0pt,fill=white] (A) at (0,0) {\textcolor{white}{$t_r$}};
\node[draw,circle,scale=0.8,inner sep=0pt,fill=white,opacity=\leop] (A) at (0,0) {$t_r$};
\begin{scope}[xshift=4pt,yshift=-4pt]
\draw[rounded corners=\roc,->,draw=\cola,line width=\arrlw] (0,-\one) --(0,0)--(\one,0);
\end{scope}
\begin{scope}[xshift=-4pt,yshift=4pt]
\draw[rounded corners=\roc,->,draw=\colb,line width=\arrlw] (\one,0) --(0,0)--(0,-\one);
\end{scope}
\end{tikzpicture}
};

\node[scale=\nodescale](A) at (\stp,0){
\begin{tikzpicture}
\draw[line width=0.8pt,opacity=\leop] (-1,0)--(1,0);
\draw[line width=0.8pt,opacity=\leop] (0,0)--(0,-1);
\node[draw=white,circle,scale=0.8,inner sep=0pt,fill=white] (A) at (0,0) {\textcolor{white}{$t_r$}};
\node[draw,circle,scale=0.8,inner sep=0pt,fill=white,opacity=\leop] (A) at (0,0) {$t_r$};

\begin{scope}[xshift=-4pt,yshift=-4pt]
\draw[rounded corners=\roc,->,draw=\cola,line width=\arrlw] (0,-\one) --(0,0)--(-\one,0);
\end{scope}

\begin{scope}[xshift=4pt,yshift=-4pt]
\draw[rounded corners=\roc,->,draw=\colb,line width=\arrlw] (\one,0) --(0,0)--(0,-\one);
\end{scope}

\begin{scope}[xshift=0,yshift=4pt]
\draw[rounded corners=\roc,->,draw=\colc,line width=\arrlw] (-\one,0)--(\one,0);
\end{scope}
\end{tikzpicture}
};

\node[scale=\nodescale](A) at (2*\stp,0){
\begin{tikzpicture}
\draw[line width=0.8pt,opacity=\leop] (0,0)--(1,0);
\draw[line width=0.8pt,opacity=\leop] (0,1)--(0,-1);
\node[draw=white,circle,scale=0.8,inner sep=0pt,fill=white] (A) at (0,0) {\textcolor{white}{$t_r$}};
\node[draw,circle,scale=0.8,inner sep=0pt,fill=white,opacity=\leop] (A) at (0,0) {$t_r$};

\begin{scope}[xshift=4pt,yshift=4pt]
\draw[rounded corners=\roc,->,draw=\cola,line width=\arrlw] (\one,0) --(0,0)--(0,\one);
\end{scope}

\begin{scope}[xshift=4pt,yshift=-4pt]
\draw[rounded corners=\roc,->,draw=\colb,line width=\arrlw]  (0,-\one)--(0,0)--(\one,0);
\end{scope}

\begin{scope}[xshift=-4pt]
\draw[rounded corners=\roc,->,draw=\colc,line width=\arrlw] (0,\one)--(0,-\one);
\end{scope}
\end{tikzpicture}
};

\node[scale=\nodescale](A) at (3*\stp,0){
\begin{tikzpicture}
\draw[line width=0.8pt,opacity=\leop] (-1,0)--(1,0);
\draw[line width=0.8pt,opacity=\leop] (0,1)--(0,-1);
\node[draw=white,circle,scale=0.8,inner sep=0pt,fill=white] (A) at (0,0) {\textcolor{white}{$t_r$}};
\node[draw,circle,scale=0.8,inner sep=0pt,fill=white,opacity=\leop] (A) at (0,0) {$t_r$};

\def\ptt{6pt}
\def\pt{3pt}
\begin{scope}[xshift=\pt,yshift=\ptt]
\draw[rounded corners=\roc,->,draw=\cola,line width=\arrlw] (\one,0) --(0,0)--(0,\one);
\end{scope}

\begin{scope}[xshift=-\ptt,yshift=-\pt]
\draw[rounded corners=\roc,->,draw=\colb,line width=\arrlw] (0,-\one) --(0,0)--(-\one,0);
\end{scope}

\begin{scope}
\draw[rounded corners=\roc,->,draw=\colc,line width=\arrlw] (-0.1,\one)--(-0.1,-\one);
\draw[rounded corners=\roc,->,draw=\cold,line width=\arrlw] (-\one,0.1)--(\one,0.1);
\end{scope}
\end{tikzpicture}
};

\end{tikzpicture}
\end{equation}
In other words, the strand $p_a$ zig-zags in the northwest direction until it hits the north or west boundary, after which it goes straight southward or straight eastward until it arrives at the boundary again. If there is no edge of $G(D)$ incident to the boundary vertex $a$ then $p_a$ is taken to be a small clockwise or counterclockwise loop depending on whether $a$ is on a vertical or horizontal edge of $\la$.  Every face $F_r$ of $G(D)$ is labeled by an $(n-k)$-element subset of $[n]$, consisting of those $a$ such that $F_r$ lies to the right of $p_a$. See \cref{fig:faces_Le} for the labeling of the Le-diagram from \cref{ex:Le_big}. From now on, we identify $F_r$ with the corresponding subset, and write $\Delta_{F_r}$ for the corresponding Pl\"ucker coordinate on $\Gr(n-k,n)$.

It is known (see~\cite[Section~5.2]{KLS}) that $F_0$ coincides with the lexicographically maximal $(n-k)$-element subset $S\subset [n]$ such that $\Delta_{S}$ is not identically zero on $\PR_v^w$. Moreover, we have $\Delta_{F_r}(x)\neq0$ for any $x\in\PR_v^w$ and any $r\in\bound\sqcup\{0\}$. Since the image of the Pl\"ucker embedding lies in the projective space, we always assume that the Pl\"ucker coordinates are rescaled (``gauge fixed'') so that $\Delta_{F_0}(x)=1$ for all $x\in \PR_v^w$.

\subsection{Main result}\label{sec:main}
Recall from \cref{sec:cluster-algebra} that the cluster algebra $\AQD$ is a subring of the field of rational functions in the variables $\{x_{F_r}\}_{r\in\Jo}$. The following result is explicitly conjectured in \cite[Remark~4.6]{MStwist}; the statement may be considered implicitly conjectured in~\cite{Pos,Sco}.
\def\ninj{\eta}
\begin{theorem}\label{thm:main}
For all $(v,w)\in Q^J$, the map sending $x_{F_r}\mapsto \Delta_{F_r}$ for each $r\in\Jo$ induces a ring isomorphism $\ninj:\AQD\xrasim \C[\PR_v^w]$ (with $\Delta_{F_0}=1$ on $\PR_v^w$).
\end{theorem}

When $v=1$ and $w$ is the maximal element of $W^J$ (i.e., when $D(v,w)$ is a $k\times (n-k)$ rectangle filled with dots), $\PR_v^w$ is the top-dimensional positroid variety in $\Gr(n-k,n)$, in which case \cref{thm:main} was shown by Scott~\cite{Sco}.

Recall that $\bound\subset \Jo$ is the set of $r\in \Jo$ such that $F_r$ labels a boundary face of $G(D)$. \Cref{thm:main} is equivalent to the following two explicit statements for Pl\"ucker coordinates on~$\PR_v^w$:
\def\clvar{x}
\begin{enumerate}
\item\label{item:main_1} We have $\ninj(\AQD)\subseteq\C[\PR_v^w]$, that is, the image $\ninj(\clvar)$ of every cluster variable $\clvar\in\AQD$ is a regular function on $\PR_v^w$. Equivalently, $\ninj(\clvar)$  can be written as a polynomial in the Pl\"ucker coordinates divided by a monomial in $\{\Delta_{F_r}\}_{r\in\bound}$.
\item\label{item:main_2} We have $\ninj(\AQD)\supseteq\C[\PR_v^w]$, that is,  the images of cluster variables generate $\C[\PR_v^w]$ \emph{as a ring}.
\end{enumerate}

In general, both of these statements are non-obvious. We will deduce~\eqref{item:main_1} from Leclerc's results in the next subsection. The non-trivial part of~\eqref{item:main_2} is that unlike in the case of the top-dimensional positroid variety~\cite{Sco}, not every Pl\"ucker coordinate is the image of a cluster variable. But \cref{thm:main} implies that every Pl\"ucker coordinate can be written as a polynomial in the images of cluster variables divided by a monomial in $\{\Delta_{F_r}\}_{r\in\bound}$.  We will prove this in \cref{sec:surj}.

\def\vf{v\cdot f}
\def\vfbar{\overline{\vf}}
\subsection{Converting Leclerc's functions into Pl\"ucker coordinates}

Let $\vf_r \in \C[\dv N_{v,w}]$ denote the image of $f_r$ under the isomorphism $\C[N_{v,w}] \simeq \C[\dv N_{v,w}]$.  Explicitly, we have $\vf_r := \Delta_{\vi r\omega_{i_r} ,\Wi{r-1}\omega_{i_r} }.$ Recall that the map $g\mapsto gB_-$ gives an isomorphism $\dv N_{v,w}\xrasim \Rich_v^w$, while the map $\pi_J:G/B_-\to G/\PJ$ restricts to an isomorphism $\Rich_v^w\xrasim \PR_v^w$. For a function $\vf\in\C[\dv N_{v,w}]$, denote by $\vfbar\in\C[\PR_v^w]$ the image of $\vf$ under the composition of these isomorphisms.

\begin{figure}
\begin{tikzpicture}

\node(X) at (6,0){
\begin{tikzpicture}[scale=1.4]
\def\nodescl{0.7}
\begin{scope}[opacity=\leop]
\draw (0,0)-- (2,0)--(2,1)--(3,1)--(3,3)--(0,3)--(0,0);
\node[anchor=north] (A) at (0.5,0) {\scalebox{\nodescl}{$1$}};
\node[anchor=north] (A) at (1.5,0) {\scalebox{\nodescl}{$2$}};
\node[anchor=north] (A) at (2.5,1) {\scalebox{\nodescl}{$4$}};
\node[anchor=west] (A) at (2,0.5) {\scalebox{\nodescl}{$3$}};
\node[anchor=west] (A) at (3,1.5) {\scalebox{\nodescl}{$5$}};
\node[anchor=west] (A) at (3,2.5) {\scalebox{\nodescl}{$6$}};
\draw[line width=0.8pt] (0.5,0.5)--(2,0.5);
\draw[line width=0.8pt] (1.5,1.5)--(3,1.5);
\draw[line width=0.8pt] (1.5,1.5)--(1.5,0);
\draw[line width=0.8pt] (2.5,2.5)--(2.5,1);
\draw[line width=0.8pt] (0.5,2.5)--(3,2.5);
\draw[line width=0.8pt] (0.5,2.5)--(0.5,0);
\end{scope}
\node[circle,scale=0.8,inner sep=0pt,fill=white] (A2) at (0.5,0.5) {$\strut$};
\node[circle,scale=0.8,inner sep=0pt,fill=white] (A1) at (1.5,0.5) {$\strut$};
\node[circle,scale=0.8,inner sep=0pt,fill=white] (A3) at (2.5,1.5) {$\strut$};
\node[circle,scale=0.8,inner sep=0pt,fill=white] (A4) at (1.5,1.5) {$\strut$};
\node[circle,scale=0.8,inner sep=0pt,fill=white] (A6) at (2.5,2.5) {$\strut$};
\node[circle,scale=0.8,inner sep=0pt,fill=white] (A8) at (0.5,2.5) {$\strut$};

\begin{scope}[opacity=\leop]
\node[draw,circle,scale=0.8,inner sep=0pt,fill=white] (B2) at (0.5,0.5) {$t_2$};
\node[draw,circle,scale=0.8,inner sep=0pt,fill=white] (B1) at (1.5,0.5) {$t_1$};
\node[draw,circle,scale=0.8,inner sep=0pt,fill=white] (B3) at (2.5,1.5) {$t_3$};
\node[draw,circle,scale=0.8,inner sep=0pt,fill=white] (B4) at (1.5,1.5) {$t_4$};
\node[draw,circle,scale=0.8,inner sep=0pt,fill=white] (B6) at (2.5,2.5) {$t_6$};
\node[draw,circle,scale=0.8,inner sep=0pt,fill=white] (B8) at (0.5,2.5) {$t_8$};
\end{scope}

\def\Fsep{2pt}
\def\Fscl{0.8}
\def\f{\textcolor{red}{5}}
\def\t{\textcolor{blue}{3}}
\node[scale=\Fscl,anchor=north west,inner sep=\Fsep] (F1) at (A1.south east) {$126$};
\node[scale=\Fscl,anchor=north west,inner sep=\Fsep] (F2) at (A2.south east) {$146$};
\node[scale=\Fscl,anchor=north west,inner sep=\Fsep] (F3) at (2.55,1.3) {$2\t4$};
\node[scale=\Fscl,anchor=north west,inner sep=\Fsep] (F4) at (1.65,1.3) {$2\t6$};
\node[scale=\Fscl,anchor=north west,inner sep=\Fsep] (F6) at (A6.south east) {$24\f$};
\node[scale=\Fscl,anchor=north west,inner sep=\Fsep] (F8) at (A8.south east) {$246$};
\node[scale=\Fscl,anchor=north west,inner sep=\Fsep] (F0) at (0.05,2.95) {$4\f6$};

\begin{scope}[yshift=0]
\draw[rounded corners=\roc,line width=\strandlw,red,->] (3,1.6)--(2.6,1.6)--(2.6,2.6)--(0.4,2.6)--(0.4,0);
\end{scope}
\begin{scope}[yshift=0]
\draw[rounded corners=\roc,line width=\strandlw,blue,->] (2,0.6)--(1.6,0.6)--(1.6,1.4)--(3,1.4);
\end{scope}

\end{tikzpicture}
};

\end{tikzpicture}
  \caption{\label{fig:faces_Le}Labeling the faces of a Le-diagram by subsets. The strand $p_5$ is shown in red and $p_3$ is shown in blue. Here we abbreviate $\{a,b,c\}$ as $abc$.}
\end{figure}
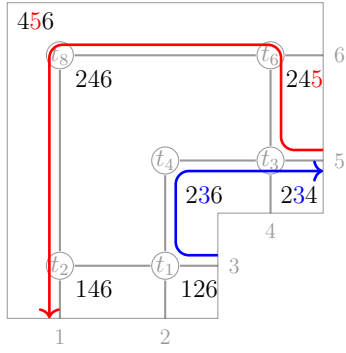

\begin{lemma}\label{lem:rotation}
Let $(v,w)\in Q^J$ and gauge-fix $\Delta_{F_0} = 1$ on $\PR_v^w$. Then for all $r\in\Jo$, the regular functions $\vfbar_r,\Delta_{F_r}\in\C[\PR_v^w]$ agree on $\PR_v^w$. 
Equivalently, we have 
\begin{equation}\label{eq:f'_r_F_r}
  f_r(g)=\frac{\Delta_{F_r}(\dv g\PJ)}{\Delta_{F_0}(\dv g\PJ)}\quad\text{for all $g\in  N_{v,w}$ and $r\in\Jo$.}
\end{equation}
\end{lemma}
\begin{proof}

We will prove~\eqref{eq:f'_r_F_r} more generally for all $g\in  N$. Observe that $F_0 = v[k+1,n]$ (see~\cite[Example~9.5]{GKL3}). Therefore for any $g\in  N$, the submatrix of $g$ with row set $v^{-1}F_0=[k+1,n]$ and column set $[k+1,n]$ is an $(n-k)\times(n-k)$ upper-triangular unipotent matrix, thus  $\Delta_{F_0}(\dv g) = 1$. It remains to show that $f_r(g)=\Delta_{F_r}(\dv g)$ for all $r\in \Jo$. 

Fix $r\in\Jo$. We have $f_r = \Delta_{A_r,B_r}$, where 
\begin{equation}\label{eq:A_B}
A_r = \Vi{r-1}\omega_{i_r}=\sv_{i_m}\cdots \sv_{i_r}\omega_{i_r}\quad\text{and}\quad B_r= \Wi{r-1}\omega_{i_r}=s_{i_m}\cdots s_{i_r}\omega_{i_r}.
\end{equation}
Recall from \cref{sec:Lec_rep} that we have two Young diagrams $\muu(\Vi{r-1},i_r)\subset \muu(\Wi{r-1},i_r)$ that fit inside an $(n-i_r)\times i_r$ rectangle, and moreover, $\muu(\Wi{r-1},i_r)$ is itself a rectangle. Thus 
there exist integers $a_r\in[0,k]$ and $b_r\in[k,n]$ such that $B_r=[1,a_r]\sqcup [k+1,b_r]$, so  $a_r+b_r-k=|B_r|=i_r$. And because $\muu(\Vi{r-1},i_r)\subset \muu(\Wi{r-1},i_r)$, we find that $[1,a_r]\subset A_r$ and $A_r\cap [b_r+1,n]=\emptyset$. Let us define $C_r:=(A_r\setminus[1,a_r])\sqcup [b_r+1,n]$, and thus $|C_r|=n-k$.  (See \cref{fig:chambers}(b) for an example.) It is clear that the functions $f_r=\Delta_{A_r,B_r}$ and $\Delta_{C_r,[k+1,n]}$ agree on $N$. Our next goal is to show
\begin{equation}\label{eq:C_r_F_r}
F_r=vC_r.
\end{equation}

\def\WD{\operatorname{WD}}
\def\WDD{\WD^\bullet}
\def\WDLE{\overline{\WDD}}
\def\pp{\overline{p}}

Let us first give a pictorial description of $A_r$ and $B_r$ using \emph{wiring diagrams}. It is analogous to~\cite[Section~9]{MR}. Draw a wiring diagram $\WD(w)$ for $w$, and for each $r\in\Jo$, place a dot labeled $t_r$ at the crossing that corresponds to $s_{i_r}$.  Denote this \emph{dotted wiring diagram} by $\WDD(v,w)$, cf. \cref{fig:chambers}(a). A wiring diagram $\WD(v)$ for $v$ is obtained from $\WDD(v,w)$ by ``uncrossing'' each dot, i.e., replacing each crossing of $\WDD(v,w)$ that has a dot by a pair of parallel wires. Label each wire in $\WD(w)$ and $\WD(v)$ by its right endpoint (the right endpoints are labeled $1,\dots, n$ from bottom to top). To each chamber $R$ of $\WD(w)$ we associate a set $B_R$  of wires that are below $R$ in $\WD(w)$. Similarly, we introduce a set $A_{R'}$ of wires that are below each chamber $R'$ of $\WD(v)$. Any chamber $R$ of $\WD(w)$ is contained inside a unique chamber $R'$ of $\WD(v)$, so we label the corresponding chamber $R$ of $\WDD(v,w)$ by the pair $(A_R,B_R)$, where $A_R:=A_{R'}$; see \cref{fig:chambers}(a). For each $r\in \Jo$, let $R_r$ be the chamber of $\WDD(v,w)$ that is immediately to the left of the dot labeled by $t_r$. It is straightforward to check (see also~\cite[Section~9]{MR}) that
\begin{equation}\label{eq:AA_BB}
  A_r=A_{R_r}\quad\text{and}\quad B_r=B_{R_r} \quad\text{for all $r\in \Jo$,}
\end{equation}
where $A_r,B_r$ are as in~\eqref{eq:A_B}.

\def\wirelw{1pt}
\def\dotscl{1}
\def\roc{5}
\def\setscl{1}
\def\rbound{19}
\def\lbound{-2.5}
\def\bnscl{1}

\def\leop{0.5}
\def\strandlw{1pt}

\def\xscl{0.4}
\def\yscl{0.7}
\begin{figure}
\scalebox{0.9}{
\begin{tabular}{c|c}
\scalebox{0.95}{
\begin{tikzpicture}[xscale=0.5,yscale=0.8,baseline=(Z.base)]
\coordinate(Z) at (0,3.5);
\def\ledot#1(#2,#3){
\xxdot{#1}(#2,#3)
\node[draw,circle,scale=\dotscl,inner sep=0pt,fill=white] (A) at (#2.5,#3.5) {$t_{#1}$};
}
\def\ex{.4}
\def\xxdot#1(#2,#3){
\fill[white] ($ (#2,#3) + (-0\ex,-0.1) $) rectangle ($ (#2,#3) + (1\ex,1.1) $);
\draw[line width=\wirelw,rounded corners=\roc] ($ (#2,#3) + (-0\ex,0) $) -- (#2,#3) --($ (#2,#3) + (1,1) $)--($ (#2,#3) + (1\ex,1) $);
\draw[line width=\wirelw,rounded corners=\roc] ($ (#2,#3) + (-0\ex,1) $) -- ($ (#2,#3) + (0,1) $) --($ (#2,#3) + (1,0) $)--($ (#2,#3) + (1\ex,0) $);
}

\foreach \i in {1,...,5,6}{
\draw[line width=\wirelw] (\lbound,\i) -- (\rbound,\i);
\node[scale=\bnscl,anchor=east] (A) at (\lbound,\i) {$\i$};
\node[scale=\bnscl,anchor=west] (A) at (\rbound,\i) {$\i$};
}
\ledot1(0,2)
\ledot2(2,1)
\ledot3(4,4)
\ledot4(6,3)
\xxdot5(8,2)
\ledot6(10,5)
\xxdot7(12,4)
\ledot8(14,3)

\def\sett(#1,#2,#3,#4){
\node[scale=\setscl,anchor=east] (A) at (#1,#2.5) {$#3,#4$};
}
\def\setx(#1,#2,#3,#4){
\node[scale=\setscl,anchor=east] (A) at ($(#1,#2.5)+(-0.5,0)$) {$#3,#4$};
}
\sett(0,2,13,45)
\sett(2,1,1,4)
\sett(4,4,1235,1456)
\sett(6,3,123,145)
\sett(8,2,13,14)
\sett(10,5,12345,12456)
\sett(12,4,1235,1245)
\sett(14,3,123,124)

\setx(\rbound,1,1,1)
\setx(\rbound,2,12,12)
\setx(\rbound,3,123,123)
\setx(\rbound,4,1234,1234)
\setx(\rbound,5,12345,12345)

\end{tikzpicture}
}
&
\def\lnn(#1,#2,#3,#4){
$#1$&$#2$&$#3$&$#4$\\\hline
}
\begin{tabular}{|l|l|l|l|}\hline
\lnn(r,A_r,B_r,C_r)
\lnn(1, 13, 45, 136)
\lnn(2, 1, 4, 156)
\lnn(3, 1235, 1456, 235)
\lnn(4, 123, 145, 236)
\lnn(6, 12345, 12456, 345)
\lnn(8, 123, 124, 356)

\end{tabular}\\
\\
\begin{tabular}{l}
(a) Labeling the regions of $\WDD(v,w)$ by $(A_r,B_r)$.
\end{tabular}
&
\begin{tabular}{l}
(b) $A_r,B_r,C_r$ for $r\in\Jo$.
\end{tabular}
\vspace{1cm}
\end{tabular}
}

\scalebox{0.9}{
\begin{tabular}{cc}

\begin{tikzpicture}[xscale=\xscl,yscale=\yscl,baseline=(Z.base)]
\draw[white] (0,3)--(-5,3);
\coordinate(Z) at (0,3.5);
\def\ledot#1(#2,#3){
\node[draw,circle,scale=\dotscl,inner sep=0pt,fill=white] (A) at (#2.5,#3.5) {$t_{#1}$};
}

\foreach \i/\rb in {1/2,2/8,3/6,4/6,5/12,6/10}{
\draw[line width=\wirelw] (\lbound,\i) -- (\rb,\i);
\node[scale=\bnscl,anchor=east] (A) at (\lbound,\i) {$\i$};
}
\draw[line width=\wirelw] (9,3) -- (14,3);
\draw[line width=\wirelw] (13,4) -- (14,4);

\def\ex{.4}

\def\whrect#1(#2,#3){
\fill[white] ($ (#2,#3) + (-0\ex,-0.1) $) rectangle ($ (#2,#3) + (1\ex,1.1) $);
}
\def\xda#1(#2,#3){
\draw[line width=\wirelw,rounded corners=\roc] ($ (#2,#3) + (-0\ex,1) $) --($ (#2,#3) + (0,1) $)--($ (#2,#3) + (0.5,0.5) $);
}
\def\xdb#1(#2,#3){
\draw[line width=\wirelw,rounded corners=\roc] ($ (#2,#3) + (-0\ex,0) $) --($ (#2,#3) + (0,0) $)--($ (#2,#3) + (0.5,0.5) $);
}
\def\xdc#1(#2,#3){
\draw[line width=\wirelw,rounded corners=\roc] ($ (#2,#3) + (1\ex,1) $) --($ (#2,#3) + (1,1) $)--($ (#2,#3) + (0.5,0.5) $);
}
\def\xdd#1(#2,#3){
\draw[line width=\wirelw,rounded corners=\roc] ($ (#2,#3) + (1\ex,0) $)--($ (#2,#3) + (1,0) $)--($ (#2,#3) + (0.5,0.5) $);
}

\def\xxd#1(#2,#3){
\xda#1(#2,#3)
\xdb#1(#2,#3)
\xdc#1(#2,#3)
\xdd#1(#2,#3)
}

\whrect1(0,2)
\whrect2(2,1)
\whrect3(4,4)
\whrect4(6,3)
\whrect6(10,5)
\whrect8(14,3)
\whrect5(8,2)
\whrect7(12,4)

\xxd1(0,2)
\xda2(2,1)\xdb2(2,1)\xdc2(2,1)
\xxd3(4,4)
\xda4(6,3)\xdb4(6,3)
\xdb5(8,2)\xdc5(8,2)
\xda6(10,5)\xdb6(10,5)\xdd6(10,5)
\xda7(12,4)\xdd7(12,4)
\xda8(14,3)\xdb8(14,3)

\ledot1(0,2)
\ledot2(2,1)
\ledot3(4,4)
\ledot4(6,3)
\ledot6(10,5)
\ledot8(14,3)

\end{tikzpicture}
&
\begin{tikzpicture}[xscale=\xscl,yscale=\yscl,baseline=(Z.base)]
\coordinate(Z) at (0,3.5);

\foreach \i/\rb in {1/2,2/8,3/6,4/6,5/12,6/10}{
\draw[line width=\wirelw,opacity=\leop] (\lbound,\i) -- (\rb,\i);
\node[scale=\bnscl,anchor=east,opacity=\leop] (A) at (\lbound,\i) {$\i$};
}
\draw[line width=\wirelw,opacity=\leop] (9,3) -- (14,3);
\draw[line width=\wirelw,opacity=\leop] (13,4) -- (14,4);

\def\ex{.4}

\def\whrect#1(#2,#3){
\fill[white] ($ (#2,#3) + (-0\ex,-0.1) $) rectangle ($ (#2,#3) + (1\ex,1.1) $);
}
\def\xda#1(#2,#3){
\draw[line width=\wirelw,rounded corners=\roc,opacity=\leop] ($ (#2,#3) + (-0\ex,1) $) --($ (#2,#3) + (0,1) $)--($ (#2,#3) + (0.5,0.5) $);
}
\def\xdb#1(#2,#3){
\draw[line width=\wirelw,rounded corners=\roc,opacity=\leop] ($ (#2,#3) + (-0\ex,0) $) --($ (#2,#3) + (0,0) $)--($ (#2,#3) + (0.5,0.5) $);
}
\def\xdc#1(#2,#3){
\draw[line width=\wirelw,rounded corners=\roc,opacity=\leop] ($ (#2,#3) + (1\ex,1) $) --($ (#2,#3) + (1,1) $)--($ (#2,#3) + (0.5,0.5) $);
}
\def\xdd#1(#2,#3){
\draw[line width=\wirelw,rounded corners=\roc,opacity=\leop] ($ (#2,#3) + (1\ex,0) $)--($ (#2,#3) + (1,0) $)--($ (#2,#3) + (0.5,0.5) $);
}

\def\xxd#1(#2,#3){
\xda#1(#2,#3)
\xdb#1(#2,#3)
\xdc#1(#2,#3)
\xdd#1(#2,#3)
}

\whrect1(0,2)
\whrect2(2,1)
\whrect3(4,4)
\whrect4(6,3)
\whrect6(10,5)
\whrect8(14,3)
\whrect5(8,2)
\whrect7(12,4)

\xxd1(0,2)
\xda2(2,1)\xdb2(2,1)\xdc2(2,1)
\xxd3(4,4)
\xda4(6,3)\xdb4(6,3)
\xdb5(8,2)\xdc5(8,2)
\xda6(10,5)\xdb6(10,5)\xdd6(10,5)
\xda7(12,4)\xdd7(12,4)
\xda8(14,3)\xdb8(14,3)

\def\ledot#1(#2,#3){
\node[draw=white,circle,scale=\dotscl,inner sep=0pt,fill=white] (A) at (#2.5,#3.5) {$\textcolor{white}{t_{#1}}$};
\node[opacity=\leop,draw,circle,scale=\dotscl,inner sep=0pt,fill=white] (A) at (#2.5,#3.5) {$t_{#1}$};
}

\ledot1(0,2)
\ledot2(2,1)
\ledot3(4,4)
\ledot4(6,3)
\ledot6(10,5)
\ledot8(14,3)

\begin{scope}[yshift=3pt]
\draw[rounded corners=\roc,line width=\strandlw,red,->] (\lbound,5) -- (12,5)--(13,4)--(15.3,4);
\draw[rounded corners=\roc,line width=\strandlw,red,->] (15.3,4)--(15,4)--(14,3)--(9,3)--(8,2)--(3,2)--(2,1)--(\lbound,1);
\end{scope}
\begin{scope}[yshift=-3pt]
\draw[rounded corners=\roc,line width=\strandlw,blue,->] (\lbound,3) -- (7.3,3);
\draw[rounded corners=\roc,line width=\strandlw,blue,->] (7.3,3) -- (7,3)--(6,4)--(5,4)--(4,5)--(\lbound,5);
\end{scope}

\end{tikzpicture}\\

\\
\begin{tabular}{l}
(c) The graph $\WDLE(v,w)$.
\end{tabular}
&
\begin{tabular}{l}
(d) Strands $\pp_4$ (red)  and $\pp_2$ (blue) in $\WDLE(v,w)$.
\end{tabular}
\end{tabular}
}

\caption{\label{fig:chambers} Constructing the graphs $\WDD(v,w)$ and $\WDLE(v,w)$ from the proof of \cref{lem:rotation}. Here $k=3$, $n=6$, and $(v,w)$ is as in \cref{ex:Le_big}.}
\end{figure}

Let us now introduce a certain planar graph $\WDLE(v,w)$ drawn in a disk, with boundary vertices labeled by $1,\dots,n$ and interior vertices labeled by $t_r$ for $r\in\Jo$. The same construction appears in~\cite[Figure~5]{Karpman2}. The graph $\WDLE(v,w)$ is obtained from $\WDD(v,w)$ by removing the \emph{redundant part} of each wire, that is, the part to the right of the rightmost dot that is placed on an intersection involving this wire; see \cref{fig:chambers}(c). (The redundant part of each wire is common to $\WD(v)$ and $\WD(w)$.)  

Observe that there is a simple isomorphism between the graphs $G(D)$ and $\WDLE(v,w)$ that preserves the labels of the vertices (i.e., $1,\dots, n$ for boundary vertices and $\{t_r\}_{r\in \Jo}$ for interior vertices). This isomorphism can be obtained by reflecting $G(D)$ along the line $y=2x$ in the $xy$-plane. For example, compare \cref{fig:chambers}(c) with \cref{fig:Le_quiver}(left). 

For each $a\in [n]$, we introduce a path $\pp_a$ in $\WDLE(v,w)$ that starts at $v^{-1}(a)$ on the left and ends at $w^{-1}(a)$ on the left. First, consider a path $\pp'_a$ in $\WDD(v,w)$ that starts at $v^{-1}(a)$ on the left, goes right following the strands of $\WD(v)$ (i.e., ignores all intersections that have dots on them) until it reaches $a$ on the right, and then goes left following the strands of $\WD(w)$ until it reaches $w^{-1}(a)$ on the left.  The path $\pp'_a$ in $\WDD(v,w)$ travels right and then left along the redundant part of the wire whose right endpoint is $a$. We define $\pp_a$ to be the path in $\WDLE(v,w)$ obtained from $\pp'_a$ by removing this redundant part. See \cref{fig:chambers}(d) for an example.  Comparing the ``rules of the road''~\eqref{eq:rules_road} with the definition of the paths $\pp_a$ in $\WDLE(v,w)$, we find that for each $a\in[n]$, our graph isomorphism $G(D)\cong \WDLE(v,w)$ sends the path $p_a$ in $G(D)$ to the path $\pp_{v(a)}$ in $\WDLE(v,w)$. For example, compare \cref{fig:chambers}(d) with \cref{fig:faces_Le}.

Note that for each $r\in\Jo$, the chamber $R_r$ of $\WDD(v,w)$ is contained inside a unique chamber (also denoted $R_r$) of $\WDLE(v,w)$. We claim that for each $a\in[n]$ and $r\in \Jo$, 
\begin{equation}\label{eq:C_r_R_r}
\text{$a$ belongs to $C_r$ if and only if the chamber $R_r$ is to the \emph{left} of the path $\pp_a$.}
\end{equation}
To show this, suppose first that $a\leq k$. Then 
\[a\in C_r\Longleftrightarrow a\in A_r\setminus [1,a_r] \Longleftrightarrow a\in A_r\setminus B_r\qquad \text{(for $a\leq k$)}.\]
On the other hand, $a\leq k$ implies $v^{-1}(a)\leq w^{-1}(a)$, so $a$ belongs to $A_r\setminus B_r$ if and only if the wire labeled $a$ in $\WD(v)$ (resp., in $\WD(w)$) is below (resp., above) the chamber $R_r$, which is equivalent to $R_r$ being to the left of the path $\pp_a$. 

Suppose now that $a\geq k+1$. Then 
\[a\in C_r\Longleftrightarrow a\in A_r\sqcup [b_r+1,n] \Longleftrightarrow \text{$a\in (A_r\cap B_r)$ or $a\notin (A_r\cup B_r)$} \qquad \text{(for $a\geq k+1$)}.\]
On the other hand, $a\geq k+1$ implies $v^{-1}(a)\geq w^{-1}(a)$, so $a$ belongs to $A_r\cap B_r$ if and only if the chamber $R_r$ is above both wires of $\pp'_a$, in which case $R_r$ is to the left of $\pp_a$. The only other case when $R_r$ is to the left of $\pp_a$ is when $R_r$ is below both wires of $\pp'_a$, and this corresponds precisely to $a\notin (A_r\cup B_r)$. This shows~\eqref{eq:C_r_R_r}. Combining~\eqref{eq:C_r_R_r} with the rule for face labels in \cref{sec:face-labels}, we obtain a proof of~\eqref{eq:C_r_F_r}.

We now deduce~\eqref{eq:f'_r_F_r} from~\eqref{eq:C_r_F_r}. For $g\in N_{v,w}$, the right hand side of~\eqref{eq:f'_r_F_r} is given by $\Delta_{vC_r,[k+1,n]}(\dv g)$, which clearly equals $f_r(g)=\Delta_{C_r,[k+1,n]}(g)$ up to sign.  %
To see that the sign is correct, observe that since $C_r$ is a face label of a Le-diagram (as we have shown above), it satisfies the following property:
\begin{equation}\label{eq:strands_property}
  \text{for all $i<j$ such that $v(i)>v(j)$ and $j\in C_r$, we have $i\in C_r$.}
\end{equation}
If~\eqref{eq:strands_property} holds for $v$, it also holds for $(\Vi r)^{-1}=\sv_{i_{r+1}}\cdots \sv_{i_m}$ for each $0\leq r\leq m-1$. The desired statement $\Delta_{vC_r,[k+1,n]}(\dv g)=\Delta_{C_r,[k+1,n]}(g)$ now follows by induction on $r=m-1,m-2,\dots,0$ in a straightforward fashion.
\end{proof}

In view of \cref{cor:irred},  Leclerc's result~\cite[Theorem~4.5]{Lec} implies in the case $(v,w)\in Q^J$  that the map $x_r\mapsto \vfbar_r$ extends to an injective ring homomorphism $\AQT\hookrightarrow \C[\PR_v^w]$. He conjectured that this map is actually an isomorphism. Thus \cref{thm:main} confirms his conjecture in the case $(v,w)\in Q^J$.

Combining Proposition \ref{prop:Lecagree} and Lemma \ref{lem:rotation}, we have the following result.
\begin{corollary}\label{cor:same_Lec}
Let $(v,w)\in Q^J$ and assume $\Delta_{F_0} = 1$ on $\PR_v^w$.
\begin{theoremlist}
\item\label{cor:same} The cluster structure of~\cite{Lec} coincides with that of \cref{thm:main}.
\item\label{cor:Lec} We have an injection $\ninj:\AQD \hookrightarrow \C[\PR_v^w]$ sending $x_{F_r}$ to $\Delta_{F_r}$ for all $r\in\Jo$.
\end{theoremlist}
\end{corollary}

\def\gt{\gbf_{\bv,\bw}(\bt)}

\section{Surjectivity}\label{sec:surj}
In view of \cref{cor:Lec}, in order to complete the proof of Theorem \ref{thm:main}, it suffices to show that the map $\ninj:\AQD \hookrightarrow \C[\PR_v^w]$ is surjective.

\subsection{Paths in Le-diagrams}\label{sec:paths}

The space $\Rich_v^w$ contains a distinguished torus of the same dimension, called the \emph{open Deodhar stratum}.  We describe a parametrization of this torus following \cite{MR}.

For $\bt=(t_r)_{r\in\Jo} \in (\C^\ast)^{|\Jo|}$, define an element 
\begin{equation}\label{eq:MR}
\gt  =  g_1\cdots g_m \in N\dv\cap B_-wB_-
\qquad \text{where} \qquad
g_r = \begin{cases} \ds_{i_r} & \mbox{if $r \notin \Jo$,} \\
x_{i_r}(t_r) &\mbox{if $r \in \Jo$.} \end{cases}
\end{equation}
The map $(\C^\ast)^{\Jo} \to \Rich_v^w$ given by $\t \mapsto \gt B_-$ is an isomorphism onto its image, the open Deodhar stratum in $\Rich_v^w$.

\def\mymat#1#2{
\scalebox{0.8}{$\left(
\begin{array}{ccc}
#1
\end{array}
\right)#2
$}
}

\def\mymatt#1#2{
\scalebox{0.8}{$\left(
\begin{array}{ccccc}
#1
\end{array}
\right)#2
$}
}

\def\myeq#1#2#3{
$#1=$\mymat{#2}{#3}
}

Let $\X^w = (B_- \dw B_-)/B_-$ be a \emph{Schubert cell} inside $G/B_-$. We have an isomorphism ${N_-\dw\cap \dw N}\xrasim \X^w$ sending $g\mapsto gB_-$. Let $\phi_w:\X^w\to N_-\dw\cap \dw N$ denote the inverse of this isomorphism.  Since $\Rich_v^w \subset \X^w$, for each $\bt\in (\C^\ast)^{\Jo}$, we have a unique $h := \phi_w(\gt B_-) \in N_- \dw \cap \dw N$ satisfying $hB_- = \gt B_-$. When $(v,w)\in Q^J$, computing the matrix $h$ amounts to computing the column-echelon form of $\gt \PJ\in \Gr(n-k,n)$. Our goal is to describe the entries of $h$ in terms of the variables $\bt$. The answer essentially coincides with the \emph{boundary measurement map} of~\cite[Definition~4.7]{Pos}.

\def\Meas{\operatorname{Meas}}
\let\wtt\wt
\def\wt(#1){\wtt_{#1}(\bt)}
\def\eps(i,j){\epsilon_{i,j}}
\def\inv(i,j){{\operatorname{inv}_{i,j}}}
\def\Gvec{\vec G}
Let $\Gvec(D)$ be obtained from $G(D)$ by orienting every vertical edge down and every horizontal edge left. 
Suppose that $i\in w[k+1,n]$ (resp., $j\in w[k]$) labels a horizontal (resp., vertical) boundary edge of $\la$. For $r\in\Jo$ and a directed path $P$ in $\Gvec(D)$, we write $r\in P$ if $P$ passes through the vertex labeled $t_r$, and let $\wt(P):=\prod_{r\in P} t_r^{-1}$. Denote $\Meas_{i,j}(\bt):=\sum_P \wt(P)$, where the sum is taken over all directed paths in $\Gvec(D)$ connecting $i$ to $j$. Finally, for $i,j\in[n]$, set $\inv(i,j):=\#\{j'>j:w(j')<i\}$, so that when $i=w(j)$, the $(i,j)$-th entry of $\dw$ equals $(-1)^{\inv(i,j)}$. The following result can be deduced from~\cite[Theorem~5.10]{TW}. We include a proof here for completeness.

\def\hr{|h]}
\def\hp{h'}
\def\gtp{\gbf'_{\bv',\bw'}(\bt)}

\def\u{y}
\def\DD{d}

\begin{proposition}\label{prop:Lepath}
Let $h = \phi_w(\gt B_-)$, where $\gt $ is as in \eqref{eq:MR}.  For $i\in w[k]$ and $j\in [k+1,n]$, the $(i,j)$-th entry of $h$ equals $(-1)^{\inv(i,j)}\Meas_{i,w(j)}(\bt)$.
\end{proposition}
\begin{proof}
Because $w\in W^J$ and $h\in N_-\dw\cap \dw N$, the left $k$ columns of $h$ coincide with the left $k$ columns of $\dw$, so we are interested in the right $n-k$ columns of $h$, which contain the identity submatrix with row set $w[k+1,n]$. Let us denote by $|h]$ the submatrix of $h$ with column set $[k+1,n]$.

We proceed by induction on the length $m = \ell(w)$ of $w$.  The case $m = 0$ is clear: the matrix $|h]$ has 0-s in all entries except for the identity matrix in the rows $k+1,k+2,\ldots,n$.  Suppose the result is known for $(v',w')$, where $w' = s_i w < w$, $v' = \sv_i v\leq v$, and $i:=i_1$. Then $D=D(v,w)$ is obtained from $D':=D(v',w')$ by adding a box (which may or may not contain a dot), whose horizontal and vertical boundary edges are labeled by $i$ and $i+1$.  Let $\gtp = g_2 g_3 \cdots g_m$ and $\hp:=\phi_{w'}(\gtp B_-)$.  

Suppose that $g_1 = \ds_i$. Then $v'=s_iv<v$, $\gt  = \ds_i \gtp$, $h = \ds_i \hp$, and the extra box of $D$ does not contain a dot. The definition of a Le-diagram implies that either there are no paths involving $i$ or no paths involving $i+1$ in $D'$.  The paths in $D$ are thus in bijection with the paths in $D'$ with the roles of $i$ and $i+1$ swapped. This agrees with $h = \ds_i \hp$, and the signs of the entries change in accordance with $(-1)^{\inv(i,j)}$.

Suppose that $g_1 = x_i(t_1)$. Then $v'=v$, $\gt  = x_i(t_1) \gtp$, and we have $|h] =  x_i(t_1)|\hp] \DD\u$, where $\DD=\diag(d_{k+1},\dots,d_n)$ is an $(n-k)\times (n-k)$ diagonal matrix and $\u = (\u_{ab})_{k+1\leq a,b\leq n}$ is an $(n-k)\times (n-k)$ lower-triangular unipotent matrix given by
\[
d_a=
\begin{cases}
  1/t_1,&\text{if $a=w^{-1}(i)$,}\\
  1,&\text{otherwise;}
\end{cases}\quad 
\u_{ab} = \begin{cases} -\hp_{i,b} & \mbox{if $a= w^{-1}(i)$ and $k+1\leq b <a$},\\
1 & \mbox{if $a=b$}, \\
0 & \mbox{otherwise.}
\end{cases}
\]
Since $\DD\u$ is lower-triangular, we have $hB_- = \gt B_-$. Multiplying by $\DD\u$ on the right ``kills off'' all nonzero entries corresponding to non-inversions of $w$ and yields $h\in N_-\dw\cap \dw N$, thus $h = \phi_w(\gt )$. Note also that the extra box of $D$ contains a dot labeled by $t_1$. Thus the matrix entries of $h$ correspond again exactly to paths in $G(D)$, and the sign of each entry agrees with $(-1)^{\inv(i,j)}$.
\end{proof}

\begin{example}\label{ex:h_paths}
Let $(v',w') = (s_2,s_2s_1s_4s_3s_2)$. We find
\[\gtp=x_2(t_2)x_1(t_3)x_4(t_4)x_3(t_5)\ds_2=\smat{
1 & 0 & t_{3} & 0 & 0 \\
0 & -t_{2} & 1 & t_{2} t_{5} & 0 \\
0 & -1 & 0 & t_{5} & 0 \\
0 & 0 & 0 & 1 & t_{4} \\
0 & 0 & 0 & 0 & 1
},\quad \hp=\smat{
0 & 0 & 1 & 0 & 0 \\
0 & 0 & 0 & 1 & 0 \\
1 & 0 & -\frac{1}{t_{2} t_{3}} & \frac{1}{t_{2}} & 0 \\
0 & 0 & 0 & 0 & 1 \\
0 & -1 & \frac{1}{t_{2} t_{3} t_{4} t_{5}} & -\frac{1}{t_{2} t_{4} t_{5}} & \frac{1}{t_{4}}
}.\]

The matrices $\gtp$ and $\hp=\phi_{w'}\left(\gtp\PJ\right)$ represent the same element of $G/B_-$, which can be checked by comparing their right-justified \emph{flag minors}: for each $j\in[n]$, the linear span of the last $j$ columns of $\gtp$ equals that of $\hp$.

Let now $v=v'$, and $w=s_3w'$. Thus $\Gvec(D')$ and $\Gvec(D)$ are given by
\begin{equation}\label{eq:GD_GDp}
\begin{tikzpicture}[baseline=(ZZ.base)]
\coordinate(ZZ) at (0,0);
\def\Lescl{0.8}
\node(X) at (-1,0){
\scalebox{\Lescl}{
\begin{tikzpicture}[baseline=(Z.base)]
  \coordinate (Z) at (0,0.5);
\def\nodescl{0.7}
\draw (0,0)-- (0,2)--(3,2)--(3,1)--(2,1)--(2,0)--(0,0);
\node[anchor=north] (A) at (0.5,0) {\scalebox{\nodescl}{$1$}};
\node[anchor=north] (A) at (1.5,0) {\scalebox{\nodescl}{$2$}};
\node[anchor=north] (A) at (2.5,1) {\scalebox{\nodescl}{$4$}};
\node[anchor=west] (A) at (2,0.5) {\scalebox{\nodescl}{$3$}};
\node[anchor=west] (A) at (3,1.5) {\scalebox{\nodescl}{$5$}};
\node[draw,circle,scale=0.8,inner sep=0pt,fill=white] (A3) at (0.5,0.5) {$t_3$};
\node[draw,circle,scale=0.8,inner sep=0pt,fill=white] (A2) at (1.5,0.5) {$t_2$};
\node[draw,circle,scale=0.8,inner sep=0pt,fill=white] (A4) at (2.5,1.5) {$t_4$};
\node[draw,circle,scale=0.8,inner sep=0pt,fill=white] (A5) at (1.5,1.5) {$t_5$};
\draw[line width=0.8pt,->,>=stealth] (2,0.5)--(A2);
\draw[line width=0.8pt,->,>=stealth] (3,1.5)--(A4);
\draw[line width=0.8pt,->,>=stealth] (A2)--(1.5,0);
\draw[line width=0.8pt,->,>=stealth] (A3)--(0.5,0);
\draw[line width=0.8pt,->,>=stealth] (A4)--(2.5,1);
\draw[line width=0.8pt,->,>=stealth] (A2)--(A3);
\draw[line width=0.8pt,->,>=stealth] (A5)--(A2);
\draw[line width=0.8pt,->,>=stealth] (A4)--(A5);
\end{tikzpicture},
}
};
\node[anchor=east,scale=1] (XX) at (X.west) {$\Gvec(D')=$}; 

\node(R) at (5,0){
\scalebox{\Lescl}{
\begin{tikzpicture}[baseline=(Z.base)]
  \coordinate (Z) at (0,0.5);
\def\nodescl{0.7}
\draw (0,0)-- (3,0)--(3,2)--(0,2)--cycle;
\node[anchor=north] (A) at (0.5,0) {\scalebox{\nodescl}{$1$}};
\node[anchor=north] (A) at (1.5,0) {\scalebox{\nodescl}{$2$}};
\node[anchor=north] (A) at (2.5,0) {\scalebox{\nodescl}{$3$}};
\node[anchor=west] (A) at (3,0.5) {\scalebox{\nodescl}{$4$}};
\node[anchor=west] (A) at (3,1.5) {\scalebox{\nodescl}{$5$}};
\node[draw,circle,scale=0.8,inner sep=0pt,fill=white] (A1) at (2.5,0.5) {$t_1$};
\node[draw,circle,scale=0.8,inner sep=0pt,fill=white] (A3) at (0.5,0.5) {$t_3$};
\node[draw,circle,scale=0.8,inner sep=0pt,fill=white] (A2) at (1.5,0.5) {$t_2$};
\node[draw,circle,scale=0.8,inner sep=0pt,fill=white] (A4) at (2.5,1.5) {$t_4$};
\node[draw,circle,scale=0.8,inner sep=0pt,fill=white] (A5) at (1.5,1.5) {$t_5$};
\draw[line width=0.8pt,->,>=stealth] (3,0.5)--(A1);
\draw[line width=0.8pt,->,>=stealth] (3,1.5)--(A4);
\draw[line width=0.8pt,->,>=stealth] (A2)--(1.5,0);
\draw[line width=0.8pt,->,>=stealth] (A3)--(0.5,0);
\draw[line width=0.8pt,->,>=stealth] (A1)--(2.5,0);
\draw[line width=0.8pt,->,>=stealth] (A2)--(A3);
\draw[line width=0.8pt,->,>=stealth] (A5)--(A2);
\draw[line width=0.8pt,->,>=stealth] (A4)--(A5);
\draw[line width=0.8pt,->,>=stealth] (A1)--(A2);
\draw[line width=0.8pt,->,>=stealth] (A4)--(A1);
\end{tikzpicture}.
}};
\node[anchor=east,scale=1] (RR) at (R.west) {$\Gvec(D)=$}; 
\end{tikzpicture}
\end{equation}

\noindent We see that the entries of $\hp$ are indeed expressed as sums over paths in $\Gvec(D')$.  Next, temporarily denoting the non-trivial entries of $\hp$ by $a,b,c,d,e$ (with $a:=\frac1{t_2t_3},\dots,e:=\frac1{t_4}$), the calculation of $h = \phi_w\left( x_3(t_1) \gtp \PJ\right)$ in the proof of Proposition \ref{prop:Lepath} proceeds as follows:

\[|\hp]=\smat{
1 & 0 & 0 \\
0 & 1 & 0 \\
-a & b & 0 \\
0 & 0 & 1 \\
c & -d & e
},\quad  x_3(t_1)|\hp]=\smat{
1 & 0 & 0 \\
0 & 1 & 0 \\
-a & b & t_1 \\
0 & 0 & 1 \\
c & -d & e},\quad x_3(t_1) |\hp] \DD=\smat{
1 & 0 & 0 \\
0 & 1 & 0 \\
-a & b & 1 \\
0 & 0 & \frac1{t_1} \\
c & -d & \frac{e}{t_1}},\]
\[  |h]=x_3(t_1) |\hp] \DD  \u=\smat{
1 & 0 & 0 \\
0 & 1 & 0 \\
0 & 0 & 1 \\
\frac a{t_1} & -\frac{b}{t_1} & \frac{1}{t_1} \\
c+\frac{ae}{t_1} & -\left(d+\frac{be}{t_1}\right) & \frac{e}{t_1}
}=\smat{
1 & 0 & 0 \\
0 & 1 & 0 \\
0 & 0 & 1 \\
\frac{1}{t_{1} t_{2} t_{3}} & -\frac{1}{t_{1} t_{2}} & \frac{1}{t_{1}} \\
\frac{1}{t_{2} t_{3} t_{4} t_{5}}+\frac{1}{t_{1} t_{2} t_{3} t_{4}} & -\left(\frac{1}{ t_{2} t_{4} t_{5}}+\frac{1}{t_{1} t_{2} t_{4}}\right) & \frac{1}{t_{1} t_{4}}
}.      \]
We indeed see that the entries of $h$ are given by sums over directed paths in $\Gvec(D)$.
\end{example}

\subsection{Muller--Speyer twist}\label{sec:twist}
\def\twist{{\tau}}
\def\twF_#1{q_{#1}}
\def\twFp_#1{q'_{#1}}
Fix $(v,w) \in Q^J$. Extending a construction of Marsh and Scott~\cite{MaSc} for the top-dimensional positroid variety,  Muller and Speyer \cite[Section~1.8]{MStwist} have defined a right twist isomorphism $\twist: \PR_v^w \xrasim \PR_v^w$. We shall not recall the definition here, however; see \cref{ex:twist} below. For $r\in\Jo$, we let $\twF_r:=\Delta_{F_r}\circ\twist \in \C[\PR_v^w]$ denote the \emph{twisted minor} indexed by the corresponding face label of $G(D)$.

\begin{proposition}\label{prop:MS}
Let $r\in \Jo$, and suppose that the faces of $G(D)$ adjacent to $t_r$ are labeled by $F_a,F_b,F_c,F_r$ as in \cref{fig:neigh}. Let $x:=\gt $ be given by~\eqref{eq:MR}.   Then with $\Delta_{F_0}=1$,
\begin{equation}\label{eq:four}
t_r = \frac{\twF_a(x)\twF_b(x)}{\twF_c(x)\twF_r(x)} \qquad \text{and} \qquad \twF_r(x) =  \prod_{r'}\frac{1}{t_{r'}},\end{equation}
where the product is taken over all $r'\in\Jo$ such that the vertex of $G(D)$ labeled $t_{r'}$ is northwest of~$F_r$.
\end{proposition}
\begin{proof}
We associate to the vertex-labeled graph $G(D)$ a planar, bipartite, edge-weighted graph $N(\t)$ via the following local substitution at each vertex $t_r$ of $G(D)$:
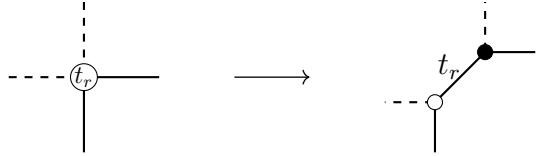
\begin{center}
\begin{tikzpicture}
\draw[dashed, line width=0.8pt] (0,1)--(0,0);
\draw[dashed, line width=0.8pt] (-1,0)--(0,0);
\draw[line width=0.8pt] (0,0)--(1,0);
\draw[line width=0.8pt] (0,0)--(0,-1);
\node[draw,circle,scale=0.8,inner sep=0pt,fill=white] (A) at (0,0) {$t_r$};
\draw[line width = 0.5pt,->] (2,0) -- (3,0);
\begin{scope}[shift={(5,0)}]
\node[draw,circle,scale=0.8,inner sep=0.5pt,fill=white] (A) at (-1/3,-1/3) {};
\node[draw,circle,scale=0.8,inner sep=0pt,fill=black] (B) at (1/3,1/3) {};
\draw[line width=0.8pt] (A) -- (B);
\draw[line width=0.8pt] (A)--(-1/3,-1);
\draw[line width=0.8pt,dashed] (A)--(-1,-1/3);
\draw[line width=0.8pt,dashed] (B)--(1/3,1);
\draw[line width=0.8pt] (B)--(1,1/3);
\filldraw[fill=white] (A) circle (0.1);
\filldraw[fill=black] (B) circle (0.1);
\node at (-0.14,0.14) {$t_r$};
\end{scope}
\end{tikzpicture}
\end{center}
Here the horizontal (resp., vertical) dashed edge is present in $N(\t)$ if and only if it is present in $G(D)$, and the weights of all horizontal and vertical edges in $N(\t)$ are set to $1$. We make the following observations concerning $N(\t)$:
\begin{enumerate}[label=(\alph*)]
\item $N(\t)$ is a \emph{reduced plabic graph} in the language of~\cite{Pos}. It satisfies the assumptions of~\cite[Section~3.1]{MStwist}.
\item The strands from \cite{Pos, MStwist} agree with the strands described in \cref{sec:face-labels}.
\item The point $\gt \PJ\in\Gr(n-k,n)$ is equal to the image of $N(\t)$ in $\Gr(n-k,n)$ under the \emph{boundary measurement map} (denoted ${\mathbb D}$ in \cite{MStwist}). This has been verified in e.g.~\cite{TW} or~\cite{Karpman2}, or can be easily checked using Le-diagram induction directly from the setup of \cite{MStwist}.
\item\label{item:downstream} The \emph{downstream wedge} (\cite[Section 5]{MStwist}) of an edge of weight $t_r$ in $N(\t)$ consists precisely of the faces $F_{r'}$ to the southeast of the vertex labeled $t_r$ in $G(D)$.
\end{enumerate}
The formula for $\twF_r(x)$ in~\eqref{eq:four} follows from~\ref{item:downstream} and~\cite[Proposition~5.5 and Theorem~7.1]{MStwist}: they denote this monomial transformation by $\overrightarrow{\mathbb M}$. The formula for $t_r$ in~\eqref{eq:four} is then obtained by expressing the values $\twF_a(x)$, $\twF_b(x)$, $\twF_c(x)$,  $\twF_r(x)$ in $\t$ using the above monomial transformation.
\end{proof}

We have the following relationship between the cluster structure and the \emph{totally nonnegative Grassmannian} studied in~\cite{Lus,Pos}.
\begin{corollary}\label{cor:tnn}
Assume $\Delta_{F_0}=1$ on $\PR_v^w$. Then the following subsets of $\PR_v^w$ coincide:
\begin{enumerate}
\item\label{PRtp1} the \emph{positroid cell} $\PRtp_v^w:= \left\{x\in\PR_v^w\mid \Delta_{I}(x)\in\R_{\geq0}\ \text{for all  $I\subset[n]$ of size $n-k$}\right\}$;
\item\label{PRtp2} the subset of $\PR_v^w$ where all cluster variables of $\AQD$ take positive real values.
\end{enumerate}
\end{corollary}
\begin{proof}
Since $\{\Delta_{F_r}\}_{r\in\Jo}$ is the image (under $\ninj$) of a single cluster of $\AQD$, and since the mutation rule~\eqref{eq:x_mut} is subtraction-free, the subset in~\eqref{PRtp2} equals $\left\{x\in\PR_v^w\mid \Delta_{F_r}(x)\in\R_{>0}\ \forall r\in \Jo\right\}$. But then applying the twist of~\cite{MStwist}, we see that this set coincides with the image of the boundary measurement map ${\mathbb D}$ applied to the graph $N(\bt)$ when $\bt$ takes values in $\R_{>0}^{\Jo}$, and this set coincides with $\PRtp_v^w$ by either~\cite{Pos} or~\cite[Section~12]{MR}.\end{proof}

\begin{remark}
For arbitrary $v\leq w\in W$, there exists a simple automorphism $\twist_{v,w}:\Rich_v^w\xrasim \Rich_v^w$ which gives a common generalization of the twist maps of~\cite{BFZ} (when $v=1$) and~\cite{MStwist} (when $(v,w)\in Q^J$), and shares many properties with them. For example, $\twist_{v,w}$ preserves the positive part $\Rtp_v^w$ of $\Rich_v^w$ and satisfies a generalization of the \emph{chamber ansatz}~\eqref{eq:four}. The map $\twist_{v,w}$ will be studied in a separate paper.
\end{remark}

\begin{example}\label{ex:twist}
Let $(v',w') = (s_2,s_2s_1s_4s_3s_2)$ and $x:=\gtp$ be  as in \cref{ex:h_paths}. Then $x$ and its twist $\twist(x)$ are represented by the following $n\times (n-k)$ matrices:
\begin{equation}\label{eq:x_twist}
|x]=\smat{
t_{3} & 0 & 0 \\
1 & t_{2} t_{5} & 0 \\
0 & t_{5} & 0 \\
0 & 1 & t_{4} \\
0 & 0 & 1
},\quad |\twist(x)]=\smat{
\frac{1}{t_{3}} & -\frac{1}{t_{2} t_{3} t_{5}} & \frac{1}{t_{2} t_{3} t_{4} t_{5}} \\
1 & 0 & 0 \\
0 & \frac{1}{t_{5}} & -\frac{1}{t_{4} t_{5}} \\
0 & 1 & 0 \\
0 & 0 & 1
}.
\end{equation}
For each $i\in[n]$, the $i$-th row of $|\twist(x)]$ is orthogonal to the $(i+1)$-th row of $|x]$, and has scalar product $1$ with the $i$-th row of $|x]$, in agreement with~\cite[Section~1.8]{MStwist}. For $G(D')$ as in~\eqref{eq:GD_GDp}, we have $F_0=\{2,4,5\}$, and we see that the matrix $|\twist(x)]$ above is gauge-fixed to have $\Delta_{F_0}=1$. The values of $\twF_r(x)=\Delta_{F_r}(\twist(x))$ for $r\in\Jo\sqcup\{0\}$ are given by:
\def\ln#1#2#3#4#5#6{$#1$&$#2$&$#3$&$#4$&$#5$&$#6$\\}
\begin{equation}\label{eq:twist_F_r_table}
\renewcommand*{\arraystretch}{1.2}
\begin{tabular}{c|c|c|c|c|c}
\ln{r}{0}{2}{3}{4}{5}\hline
\ln{F_r}{\{2,4,5\}}{\{1,2,5\}}{\{1,4,5\}}{\{2,3,4\}}{\{2,3,5\}}\hline
\ln{\twF_r(x)}{1}{\frac1{t_2t_3t_5}}{\frac1{t_3}}{\frac1{t_4t_5}}{\frac1{t_5}}
\end{tabular}
\end{equation}
This agrees with~\eqref{eq:four}.
\end{example}

\subsection{Proof of Theorem~\ref{thm:main}}
Our approach is similar to that of \cite{BFZ}, who gave an upper cluster algebra structure on double Bruhat cells.

\def\inj{\eta^{\parr \twist}}
Fix $(v,w) \in Q^J$.  Let $\inj:\AQD\hookrightarrow \C[\PR_v^w]$ be obtained by composing the map $\ninj$ from \cref{cor:Lec} with the twist isomorphism $\twist:\PR_v^w \xrasim \PR_v^w$. Explicitly, for $r\in\Jo$, $\inj$ sends $x_r\in\AQD$ to the element $\twF_r\in\C[\PR_v^w]$ from \cref{sec:twist}. 

\def\btw{{\mathbf{q}}}

By \cref{prop:MS}, this injection is induced by the invertible monomial transformation~\eqref{eq:four} between $\bt:=\{t_r\}_{r\in\Jo}$ and $\btw:=\{\twF_r\}_{r\in\Jo}$.  By \cref{sec:pos}, we have an injection $(\Cast)^{\Jo}\hookrightarrow \Rich_v^w$ sending $\bt\mapsto \gt B_-$, whose image is an open Zariski dense subset of $\Rich_v^w$. This gives rise to an injection 
\begin{equation}\label{eq:inj_T}
\C[\PR_v^w] \hookrightarrow \C[\bt^{\pm1}] = \C[\btw^{\pm1}],\quad\text{where \quad$\bt^{\pm1}:=\{t_r^{\pm1}\}_{r\in\Jo}$\quad and\quad $\btw^{\pm1}:=\{q_r^{\pm1}\}_{r\in\Jo}$.}
\end{equation}

\def\btwpr{\btw'_r}
\noindent For $r\in \Jo$, let $\twFp_r:=\inj(x_r')$ (where $x_r'$ is given in~\eqref{eq:x_mut}), and let $\btwpr:=\{\twF_{a}\}_{a\in\Jo\setminus\{r\}}\sqcup \{\twFp_r\}.$
\begin{lemma}\label{lem:mutate} For each $r\in\Jo$, we have an injection $\C[\PR_v^w] \hookrightarrow \C[(\btwpr)^{\pm1}]$. In other words, every element of $\C[\PR_v^w]$ can be written as a Laurent polynomial in the variables $\btwpr$.
\end{lemma}

\begin{proof}
\def\f{\theta}
\def\Tint{T}

Let $\Tint:=\Spec(\C[(\btw)^{\pm1}])$ denote the initial cluster torus, and let $T'_r := \Spec(\C[(\btwpr)^{\pm1}])$ denote the mutated cluster torus in the $r$-th direction. The intersection $\Tint\cap T'_r$ is Zariski dense in $T'_r$, thus the map~\eqref{eq:inj_T} gives a rational map $T'_r \dashrightarrow \PR_v^w$. (This is also the map induced by the inclusion $\inj:\AQD \hookrightarrow \C[\PR_v^w]$.) The required statement is equivalent to showing that this rational map $T'_r \dashrightarrow \PR_v^w$ is in fact an inclusion $T'_r \hookrightarrow \PR_v^w$.  Each $\twF_i$ and $\twFp_r$ is a regular function on $\PR_v^w$ and hence we have a regular map $\PR_v^w \to  \Spec(\C[\btwpr]) \simeq \C^{\Jo}$.  It thus suffices to show that the rational map $T'_r \dashrightarrow \PR_v^w$ is a regular map on the torus $T'_r$. (Indeed, in this case, the composition $T'_r\to \PR_v^w\to \Spec(\C[\btwpr])$ is a regular map whose restriction to the open dense subset $T'_r\cap T$ agrees with the inclusion map $T'_r \hookrightarrow \Spec(\C[\btwpr])$. Therefore this composition coincides with the identity map on $T'_r$, and in particular the map $T'_r\to \PR_v^w$ is automatically injective.)

We begin by showing that the map $\Tint\cap T'_r \hookrightarrow \Gr(n-k,n)$ given by $\t \mapsto \gt \PJ$ extends to a regular map $\f_r: T'_r \to \Gr(n-k,n)$.  It suffices to write each matrix entry of $|h]$ (where $h=\phi_w(\gt B_-)$ as in \cref{sec:paths}) as an element of $\C[(\btwpr)^{\pm1}]$.  By Proposition \ref{prop:Lepath}, each such matrix entry is a sum of $\wt(P)$ over paths $P$ in $\Gvec(D)$.  We may restrict our attention to paths $P$ such that the monomial $\wt(P)$ contains $\twF_r$ in the denominator.  Let $r_1$ (resp., $r_2$) be the bottom-left (resp., top-right) vertex of the face of $G(D)$ labeled by $F_r$.  Then $\wt(P)$ contains $\twF_r$ in the denominator precisely when $P$ passes through both $r_1$ and $r_2$, and either contains the top-left or the bottom-right boundary of the face $F_r$.  We may group such paths into pairs $(P_1,P_2)$ where $P_1$ contains the top-left boundary of $F_r$, while $P_2$ contains the bottom-right boundary of $F_r$, and otherwise $P_1$ and $P_2$ agree.  By \cref{prop:Lepath}, the contribution of such a pair is
\[
\wt(P_1) + \wt(P_2) = 
\frac{M}{\twF_r} \cdot \left( \frac{\twF_c}{\twF_a\twF_b} + \frac{ \prod_{F_r\to F_j:j\neq a,b} \twF_{j}}{ \prod_{F_i\to F_r: i\neq c} \twF_{i}} \right) = 
\frac{M}{\twF_r} \cdot \left(\frac{\twF_r \twFp_r}{ \twF_a\twF_b\prod_{F_i\to F_r: i\neq c} \twF_{i}}  \right),  
\]
where $F_a,F_b,F_c,F_r$ are the labels of the faces adjacent to $t_r$ as in \cref{fig:neigh}, $M$ is a monomial in $\{\twF_a\}_{a \in \Jo\setminus\{r\}}$, and the products in the second term are taken over the arrows of the quiver $\QD$ not involving $F_a,F_b,F_c$.  The common factor $\twF_r$ cancels, and we have constructed our desired map $\f_r: T'_r \to \Gr(n-k,n)$.

The intersection $\Tint \cap T'_r$ is dense in $T'_r$, and $\f_r(\Tint \cap T'_r) \subseteq \PR_v^w$, so we must have $\f_r(T'_r) \subseteq \PRcl_v^w$, where $\PRcl_v^w$ is the Zariski closure of $\PR_v^w$. By~\cite[Equation~(9)]{MStwist}, we have $\Delta_{F_r}=\frac1{\twF_r}$ for all $r\in \bound\sqcup\{0\}$. (For example, compare the minors $\Delta_{F_r}(x)$ and $\Delta_{F_r}(\twist(x))$ for $x,\twist(x)$ from~\eqref{eq:x_twist} and $F_r$ from~\eqref{eq:twist_F_r_table}.) Thus $\Delta_{F_r}$ is nonzero on $\f_r(T'_r)$ for any $r\in\bound\sqcup\{0\}$. 
 By \cite[Section 5]{KLS}, $\PR_v^w$ is exactly the locus in $\PRcl_v^w$ where the Pl\"ucker variables indexed by the \emph{Grassmann necklace} are nonvanishing, and this Grassmann necklace is precisely the collection $\{F_r\}_{r\in\bound\sqcup\{0\}}$; see \cite{Pos} and~\cite[Proposition~4.3]{MStwist}.  We conclude that we have a regular map $\f_r: T'_r \to \PR_v^w$.
\end{proof}

\begin{example}
\renewcommand*{\arraystretch}{1.2}
Let $k=3$, $n=6$, $(v,w)=(s_2s_4,s_2s_1s_4s_3s_2s_5s_4s_3)$ as in \cref{ex:Le_big}. Using \cref{fig:Le_quiver} and \cref{prop:MS}, we find
\def\quadsmall{\quad}
\begin{equation}\label{eq:t_q_Ex}
t_{8} = \frac{1}{\twF_{8}},\quadsmall t_{6} = \frac{\twF_{8}}{\twF_{6}},\quadsmall t_{4} = \frac{\twF_{8}}{\twF_{4}},\quadsmall t_{3} = \frac{\twF_{4} \twF_{6}}{\twF_{3} \twF_{8}},\quadsmall t_{2} = \frac{\twF_{8}}{\twF_{2}},\quadsmall t_{1} = \frac{\twF_{2} \twF_{4}}{\twF_{1} \twF_{8}}, 
\end{equation}
and the mutation rule~\eqref{eq:x_mut} gives $\twFp_8=\frac{\twF_{2} \twF_{4} \twF_{6} + \twF_{1} \twF_{3}}{\twF_8}$. We now express $h=\phi_w(\gt B_-)$ both in terms of $\bt$ and in terms of $\btw$ using \cref{prop:Lepath} and Equation~\eqref{eq:t_q_Ex}:
\[|h]=\smat{
1 & 0 & 0 \\
0 & 1 & 0 \\
-\frac{1}{t_{1} t_{2}} & \frac{1}{t_{1}} & 0 \\
0 & 0 & 1 \\
\frac{1}{t_{1} t_{2} t_{3} t_{4}} & -\frac{1}{t_{1} t_{3} t_{4}} & \frac{1}{t_{3}} \\
\frac{t_{1} t_{3} t_{4} + t_{8}}{t_{1} t_{2} t_{3} t_{4} t_{6} t_{8}} & -\frac{1}{t_{1} t_{3} t_{4} t_{6}} & \frac{1}{t_{3} t_{6}}
}=\smat{
1 & 0 & 0 \\
0 & 1 & 0 \\
-\frac{\twF_{1}}{\twF_{4}} & \frac{\twF_{1} \twF_{8}}{\twF_{2} \twF_{4}} & 0 \\
0 & 0 & 1 \\
\frac{\twF_{1} \twF_{3}}{\twF_{4} \twF_{6}} & -\frac{\twF_{1} \twF_{3} \twF_{8}}{\twF_{2} \twF_{4} \twF_{6}} & \frac{\twF_{3} \twF_{8}}{\twF_{4} \twF_{6}} \\
\frac{\twF_{2} \twF_{4} \twF_{6} + \twF_{1} \twF_{3}}{\twF_{4} \twF_{8}} & -\frac{\twF_{1} \twF_{3}}{\twF_{2} \twF_{4}} & \frac{\twF_{3}}{\twF_{4}}
}. \]

\def\btwp{\btw'}
\noindent Thus the only entry of $h$ that has $\twF_8$ in the denominator is 
\[\frac{t_{1} t_{3} t_{4} + t_{8}}{t_{1} t_{2} t_{3} t_{4} t_{6} t_{8}}=\frac{\twF_{2} \twF_{4} \twF_{6} + \twF_{1} \twF_{3}}{\twF_{4} \twF_{8}} = \frac{\twFp_8}{\twF_4}.\]
In particular, all matrix entries of $h$ can be written as Laurent polynomials in the cluster $\btwp_8=\{\twF_1,\twF_2,\twF_3,\twF_4,\twF_6,\twFp_8\}$, in agreement with \cref{lem:mutate}.
\end{example}

Before we finish the proof, we need one more technical statement. Given an ice quiver $Q$ with vertex set  $V=V_f\sqcup V_m$ partitioned as in \cref{sec:cluster-algebra}, the \emph{extended exchange matrix} $\tilde B(Q)=(b_{r,r'})_{r\in V,r'\in V_m}$ of $Q$ has rows indexed by the vertices of $Q$ and columns indexed by the mutable vertices of $Q$.  We have $b_{r,r'} \in\{ 1,-1,0\}$, depending on whether there is an arrow $r \to r'$, or an arrow $r' \to r$, or no arrows between $r$ and $r'$ (assuming no two vertices of $Q$ are connected by more than one arrow).

\begin{proposition}\label{prop:fullrank}
The  extended exchange matrix $\tilde B(\QD)$ is of full rank (i.e., has rank $|V_m|$).
\end{proposition}
\begin{proof}
We proceed by induction on the size of the Young diagram $\lambda$ of $D$, the case $|\lambda| = 0$ being trivial.  Suppose $|\lambda| >0$.  Let $D'$ be obtained from $D$ by removing a box $(i,j)$ adjacent to the boundary of $\la$.  If $D$ does not contain a dot inside the box $(i,j)$ then $\tilde B(\QD) = \tilde B(\QDp)$ so the result holds by induction.  Thus assume that $D$ contains a dot labeled $t_r$ inside the box $(i,j)$. Then $F_r$ is a boundary face. If either the row or the column of $(i,j)$ contains no other dots, then $\tilde B(\QD)$ is obtained from $\tilde B(\QDp)$ by removing the row indexed by $F_r$, and this row is 0; the result holds by induction.  Finally, suppose that both the row and the column of $(i,j)$ contains another dot.  Let $F_a,F_b,F_c,F_r$ be the labels of faces adjacent to $t_r$ as in \cref{fig:neigh}. Thus $F_a$ and $F_b$ are boundary faces. If $F_c$ is also a boundary face, then again $\tilde B(\QD)$ and $\tilde B(\QDp)$ differ by a $0$ row.  So assume that $F_c$ is an interior face, then $F_c$ becomes a boundary face in $G(D')$.  The matrix $\tilde B(\QD)$ satisfies $b_{F_r,F_c} = -1$, and this is the only nonzero entry in the row indexed by $F_r$.  The matrix $\tilde B(\QDp)$ is obtained from $\tilde B(\QD)$ by deleting the row of $F_r$ and the column of $F_c$.  It is clear that $\tilde B(\QD)$ has full rank if and only if  $\tilde B(\QDp)$ has full rank, so the result again holds by induction.
\end{proof}

\def\UCAQ{\overline{\Acal}(\QD)}

\begin{proof}[Proof of \cref{thm:main}]
By Proposition \ref{prop:fullrank} and \cite[Corollary 1.9]{BFZ}, the intersection of Laurent polynomial rings (called the \emph{upper bound} in \cite{BFZ})
$$
\C[(\btw)^{\pm1}] \cap \left( \bigcap_r \C[(\btwpr)^{\pm1}] \right)
$$
is equal to the \emph{upper cluster algebra} $\UCAQ$, defined to be the intersection of Laurent polynomial rings for \emph{all} clusters of $\AQD$.   By Lemma \ref{lem:mutate}, we have $\C[\PR_v^w] \subseteq \inj(\UCAQ)$.  By \cite{MSLA}, $\AQD$ is a locally acyclic cluster algebra, and by \cite[Theorem 4.1]{Mul}, we have $\AQD = \UCAQ$.  Recall that $\twist:\PR_v^w\xrasim\PR_v^w$ is an isomorphism  and $\inj=\twist\circ \ninj$. By \cref{cor:Lec}, we have $\inj(\AQD) \subseteq \C[\PR_v^w]$, and therefore $\inj(\AQD) = \C[\PR_v^w]$. We conclude that $\ninj(\AQD)=\C[\PR_v^w]$, completing the proof. 
\end{proof}

\bibliographystyle{alpha}
\bibliography{positroidcluster}
\end{document}